\documentclass[a4paper,10pt]{article}
\usepackage[top=2.54cm,bottom=2.0cm,left=2.0cm,right=2.54cm, includeheadfoot]{geometry}

\usepackage[T1]{fontenc}
\usepackage[utf8]{inputenc}
\usepackage[english]{babel}
\usepackage{enumerate}
\usepackage{multirow,booktabs}
\usepackage[table]{xcolor}
\usepackage{fullpage}
\usepackage{lastpage}
\usepackage{indentfirst}

\usepackage{footnote}

\usepackage{amsmath,amsfonts,amssymb,amscd,amsthm}

\usepackage{bm}

\usepackage[all,2cell]{xy} \UseAllTwocells \SilentMatrices

\usepackage{mathpartir}

\usepackage{hyperref}

\usepackage{subfiles}

\usepackage{multicol}

\usepackage{lineno}

\newtheorem{theorem}{Theorem}[section]
\newtheorem{proposition}[theorem]{Proposition} 

\newtheorem{definition}[theorem]{Definition} 
\newtheorem{example}[theorem]{Example}

\newtheorem*{claim*}{Claim}

\newtheorem{remark}[theorem]{Remark}

\begin{document}

\title{Generalized Marshall Quotients and Real Semigroups of Continuous and Differentiable Functions}

\author{
   Kaique Matias de Andrade Roberto $^\textup{\scriptsize a}$
}

\date{
   $^\textup{\scriptsize a}$\textit{\small Centre for Logic, Epistemology and the History of Science (CLE), State University of Campinas (UNICAMP), Campinas, Brazil}
}

\maketitle

\begin{abstract}
The theory of real semigroups developed by M. Dickmann and A. Petrovich provides an algebraic framework for abstract real spectra and real algebraic geometry, yet its application to rings of continuous functions is historically hindered by the topological constraints. In this paper, we bridge this gap by introducing generalized Marshall quotients over rings of real-valued continuous and differentiable functions, yielding new explicitly calculated examples of real semigroups. Furthermore, we conclude that the group of invertible elements of these quotients constitutes a real reduced hyperfield (which is categorically equivalent to reduced special groups), addressing the open problem of characterizing when the units of a real semigroup form a reduced special group. Finally, we apply this hyperalgebraic machinery to translate topological and differential phenomena into hyperalgebraic identities, establishing generalized versions of the \L{}ojasiewicz-type inequalities.


\end{abstract}

\section{Introduction}

% \linenumbers                % mostra as linhas no pdf

The theory of real semigroups, introduced by M. Dickmann and A. Petrovich in
\cite{dickmann2004real}, provides an algebraic framework for abstract real
spectra and real algebraic geometry, lifting the duality between reduced special groups and abstract ordering spaces developed in \cite{dickmann2000special}. A real semigroup is a ternary semigroup
equipped with ternary relations $D$ and $D^t$ that respectively provide first order axiomatic notions of
representation and transversal of elements by binary quadratic forms. The category of real
semigroups is isomorphic to that of real reduced multirings, a
hyperalgebraic counterpart developed by M. Marshall
\cite{marshall2006real}. This dual perspective has proved extremely fruitful:
concrete real semigroups arise naturally from rings of continuous functions,
fields, and abstract spaces of orderings, and the hyperfield structure allows one
to import algebraic tools such as Marshall's quotient and the
Positivstellensatz (as detailed in \cite{marshall2006real}, \cite{roberto2021quadratic} and \cite{ribeiro2016functorial}).

A central construction in the theory is the \emph{Marshall quotient} of a
multiring $A$ by a multiplicative subset $S$, denoted $A/_m S$. When $A$ is
the ring $\mathcal C(X,\mathbb R)$ of real-valued continuous functions on a
$T_6$ space $X$ and $S = q(A) = A^2 \cap \mbox{nzd}(A)$ is the set of
non-negative non-zero-divisors that are squares, it is known that
$A/_m q(A)$ is a real reduced multiring \cite{roberto2021quadratic}.
Thus $A/_m q(A)$ becomes a real semigroup, and using the recent results developed in \cite{miraglia2021boolean} we conclude that its invertible elements form a
reduced special group, a key link with the theory of quadratic forms.

However, the classical Marshall quotient discards only the ``local'' information encoded in the zero-divisors of $A$, and it does so globally on the whole space $X$. In many situations one is interested in the behavior of functions \emph{outside} a prescribed family of exceptional sets: for instance, asymptotic behavior at infinity (outside compact sets), generic behavior on a dense open set (outside nowhere dense closed sets), or metric asymptotic behavior (outside bounded sets). To capture such phenomena, we introduce in this paper \emph{generalized Marshall quotients}. Instead of requiring two functions to be equal globally up to a square factor, we only demand equality outside a set belonging to a fixed ideal of subsets. For practical reasons we restrict our analysis to compact sets, a given family of closed sets, or bounded sets (but the abstract framework is available for more general families of subsets). The resulting quotients, denoted $A/_{cm}q(A)$, $A/_{clm}q(A)$ and $A/_{bm}q(A)$, are shown to be real reduced multirings when $X$ satisfies mild topological hypotheses (e.g., $X$ is a non-compact $T_6$ space for the compact case, or a $T_6$ space with a suitable ideal of closed sets, or an unbounded metric space for the bounded case). Moreover, they are Boolean real semigroups, and their groups of units are reduced special groups, yielding new examples of real hyperfields that encode asymptotic, generic or metric-asymptotic properties of continuous
functions.

The second part of the article extends this construction to rings of differentiable functions $\mathcal C^k(\mathbb R^n,\mathbb R)$. The direct translation fails: a non-negative $C^k$ function need not be a square in $\mathcal C^k$, so the global Marshall quotient is not real reduced. We isolate the exact obstruction (the possible unboundedness of the zero-set) and show that by restricting to the subring $D \subseteq \mathcal C^k$ of functions that are constant outside a compact set (respectively, outside a closed set), the generalized quotients become real reduced multirings again. The compact and closed cases are treated in detail; the bounded case collapses to the compact one in Euclidean space, and the general situation is discussed in a section of comments.

With the algebraic machinery in place, we turn to applications. The
hyperalgebraic identities that hold in real reduced multirings translate, via the abstract version of \L{}ojasiewicz inequality provided by M. Marshall in Theorem 6.8.1 of \cite{marshall1996spaces}, into concrete topological and differential statements. We obtain asymptotic, generic, and metric-asymptotic versions of the \L{}ojasiewicz inequality, alongside positive extension results for fractured domains that are beyond the reach of classical Tietze-type theorems. These applications illustrate how the Boolean real semigroup structure, when applied in the context of continuous and differentiable functions, has potential to turn hard analytical problems into algebraic identities.

The paper is organized as follows. Section 2 collects the necessary preliminaries on rings, multirings, hyperfields, and real semigroups. Section 3 develops the generalized Marshall quotients for continuous functions and proves they are real reduced Boolean semigroups, concluding with the translation of Marshall's abstract \L{}ojasiewicz inequalities to concrete topological contexts. Section 4 extends this construction to differentiable functions, isolating the necessary boundedness restrictions to preserve real reducedness, and establishes smooth-at-infinity interpolations and extensions. Finally, Section 5 concludes with some final remarks and further investigations.

\section{Preliminaries}

In this Section, after establishing some notation and proving some folklore results concerning real-valued continuous functions in Subsection \ref{subsec-2.1}, we provide a really brief recap about multirings, hyperfields and real semigroups. These theories are extensively studied in the past decades, and here we will only recall the main definitions and results that will be relevant for this work. For more details, proofs and examples we kindly suggest the reader consult \cite{dickmann2000special}, \cite{dickmann2012real}, \cite{dickmann2012spectral}, \cite{dickmann2015faithfully}, \cite{dickmann2017constructions}, \cite{ribeiro2016functorial} and \cite{roberto2021quadratic}.

\subsection{Notations about Rings, continuous functions and perfectly normal spaces}\label{subsec-2.1}

In this work, every ring will be considered as a commutative ring with unity. Given a ring $A$, we denote
\begin{align*}
    A^\times&:=\{a\in A:a\mbox{ is invertible}\}=\{a\in A:\mbox{ there exists }b\in A\mbox{ with }ab=1\};\\
    A^2&:=\{a^2:a\in A\};\\
    \mbox{zd}(A)&:=\{a\in A:a\mbox{ is a zero-divisor}\}=\{a\in A:\mbox{ there exists }a\in A\mbox{ with }ab=0\};\\
    \mbox{nzd}(A)&:=A\setminus\mbox{zd}(A)=\{a\in A:a\mbox{ is not a zero-divisor}\};\\
    q(A)&:=A^2\cap\mbox{nzd}(A).
\end{align*}

In the sequel, we describe $\mbox{zd}(A)$ and $\mbox{nzd}(A)$ for $A=\mathcal C(X,\mathbb R)$, for suitable topological spaces $X$ (namely, $T_6$ spaces). Although the calculations are straightforward, to the best of our knowledge they are not covered in the literature as presented here.

\begin{definition}
    Let $(X,\tau)$ be a topological space and $V\subseteq X$. We say that $V$ is a:
    \begin{enumerate}[a -]
        \item \textbf{$G_\delta$-set} if $V$ is written as a countable union of open subsets of $X$.

        \item $F_\sigma$-set if $V$ is written as a countable intersection of closed sets.
    \end{enumerate}
\end{definition}

\begin{definition}[\cite{engelking1968outline} ]
Let $(X, \tau)$ be a topological space. The topology $\tau$ is called perfectly normal if it is normal and every closed set is $G_{\delta}$-set. The topology $\tau$ is called $T_6$ if it is Hausdorff and perfectly normal.
\end{definition}

Every metric space is $T_6$ (Corollary 4.1.13 of \cite{engelking1968outline}). 

Before we proceed, let us establish some notation about functions $f\in\mathcal C(X,\mathbb R)$ for a given topological space $X$. Given such $f$, we denote:
\begin{align*}
    Z(f)&:=\{x\in X:f(x)=0\};\\
    \mbox{supp}(f)&:=\{x\in X:f(x)\ne0\}=X\setminus Z(f)
\end{align*}

In the Proposition below, there are gathered the principal facts about $T_6$ spaces that will be used throughout this work:

% Facts about T_6 spaces (ver Engelking)
\begin{proposition}[Vendenissov, 1.5.13 and 1.5.19 of \cite{engelking1968outline}]\label{t6-carac}
    Let $X$ be a $T_1$ topological space $X$. The following conditions are equivalent:
    \begin{enumerate}
        \item The space $X$ is perfectly normal.

        \item For every open set $U$ there exists a continuous function $f: X \rightarrow \mathbb{R}$ such that $X \setminus U = f^{-1}(0)$.

        \item For every closed set $F$, there exists a continuous function $f \colon X \to \mathbb{R}$ such that $F = f^{-1}(0)$.

        \item For every pair of disjoint closed subsets $A,B\subseteq X$ there exists a continuous function $f \colon X \to \mathbb{R}$ such that $f^{-1}(\{0\})=A$ and $f^{-1}(\{1\})=B$.
    \end{enumerate}
    Moreover, if $X$ is perfectly normal, given a open set $U \subseteq X$ there exists a continuous function $g \colon X \to \mathbb{R}$ such that $g|_U$ is strictly positive and $Z(g)= U^c$.
\end{proposition}

\begin{proposition}[Example 5.15 in \cite{roberto2021quadratic}]
    Let $X$ be a $T_6$ topological space. Then
    $f\in\mbox{nzd}(\mathcal C(X,\mathbb R))$ if, and only if, the interior of $Z(f)$ is empty.
\end{proposition}
\begin{proof}
    Let $f\in \mathcal C(X,\mathbb R)$. If $U \subseteq Z(f)$ is a non-empty open set, then there exists $g \in C(X,\mathbb{R})$ such that $Z(g) = U^c$. Therefore $g$ is a non-zero function and $fg = 0$. Conversely, if $Z(f)$ has  empty interior and $g \in C(X,\mathbb{R})$ is such that $fg = 0$, then $Z(f)^c$ is open and dense while $Z(f)^c \subseteq Z(g)$. Since $g$ is continuous, $g = 0$ and so $f\in\mbox{nzd}(\mathcal C(X,\mathbb R))$.
\end{proof}

Then we have the following characterizations for $\mbox{nzd}(A)$, with $A:=\mathcal C(X,\mathbb R)$ and $X$ being a $T_6$ topological space:
\begin{align*}
    \mbox{nzd}(A)&=\{f\in A:Z(f)\mbox{ has empty interior}\};\\
    \mbox{zd}(A)&=\{f\in A:\mbox{there is a non-empty open set }U\subseteq X\mbox{ with }U\subseteq Z(f)\}.
\end{align*}

\begin{proposition}\label{nz-extension}
    Let $X$ be a $T_6$ topological space and $A:=\mathcal C(X,\mathbb R)$. If $f\in \mbox{zd}(A)$ then there exists $u\in\mbox{nzd}(A)$ such that $f|_{X\setminus Z(f)}=u|_{X\setminus Z(f)}$. Such a function $u$ will be called a \textbf{non-zero divisor continuous extension} of $f$. Moreover, if $X=\mathbb R^n$ (or more generally, a finite-dimensional Banach space) and $f$ is $C^k$, then we can choose a $C^k$ non-zero continuous extension $u$ (which will be briefly denoted by \textbf{$C^k$ non-zero divisor extension}).
\end{proposition}
\begin{proof}
    Let $f \in \mbox{zd}(A)$. By Proposition \ref{t6-carac}, the open set $U := \mbox{int}(Z(f))$ is non empty (here ``int'' means the interior operator). Then, for the closed set $F:=X\setminus U$ there exists a non-negative continuous function $g \in A$ such that $Z(g) = F$. Consequently, $g(x) = 0$ for all $x \notin U$, and $g(x) > 0$ for all $x \in U$.

    Let $u:=f-g$. If $x \in X \setminus Z(f)$, then $x \notin U$, which means $g(x) = 0$. Therefore, $u(x) = f(x) - 0 = f(x)$, proving that $f|_{X\setminus Z(f)} = u|_{X\setminus Z(f)}$. It remains to show that $u \in \mbox{nzd}(A)$, which is equivalent to showing that $\mbox{int}(Z(u)) = \emptyset$. We evaluate the zero-set of $u$:
    \begin{itemize}
        \item For $x \in X \setminus Z(f)$, $f(x) \neq 0$ and $g(x) = 0$, so $u(x) \neq 0$.
        \item For $x \in Z(f) \setminus U$, $f(x) = 0$ and $g(x) = 0$, so $u(x) = 0$.
        \item For $x \in U$, $f(x) = 0$ and $g(x) > 0$, so $u(x) > 0$.
    \end{itemize}
    Thus, $Z(u) = Z(f) \setminus U = Z(f) \setminus \mbox{int}(Z(f)) = \partial Z(f)$. Since the boundary of any closed set has an empty interior, $\mbox{int}(Z(u)) = \emptyset$, concluding that $u \in \mbox{nzd}(A)$.

    Suppose now that $X = \mathbb{R}^n$ (or, more generally, a finite-dimensional Banach space, which is linearly homeomorphic to $\mathbb{R}^n$). Let $f \in \mathcal{C}^k(X, \mathbb{R})$ be a zero-divisor. The set $F = X \setminus \mbox{int}(Z(f))$ is closed in $X$. A classical consequence of the existence of smooth bump functions in Euclidean spaces (see, for example, Section 2.4 of \cite{lee2000introduction}), guarantees that any closed set in $\mathbb{R}^n$ is the exact zero-set of some non-negative smooth function. Thus, there exists $g \in \mathcal{C}^\infty(X, \mathbb{R}) \subseteq \mathcal{C}^k(X, \mathbb{R})$ such that $g \ge 0$ and $Z(g) = F$.

    Defining $u := f - g$, we have that $u \in \mathcal{C}^k(X, \mathbb{R})$, $u$ coincides with $f$ outside $Z(f)$, and $Z(u) = \partial Z(f)$. Since $\mbox{int}(Z(u)) = \emptyset$, $u$ is a $\mathcal{C}^k$ non-zero extension of $f$.
\end{proof}

\subsection{Multirings and hyperfields}

\begin{definition}[Adapted from Definition 2.1 in \cite{marshall2006real}]\label{defn:multiring}
 A multiring is a sextuple $(R,+,\cdot,-,0,1)$ where $R$ is a non-empty set, $+:R\times 
R\rightarrow\mathcal P(R)\setminus\{\emptyset\}$,
 $\cdot:R\times R\rightarrow R$
 and $-:R\rightarrow R$ are functions, $0$ and $1$ are elements of $R$ satisfying:
 \begin{enumerate}[i -]
  \item $(R,+,-,0)$ is a commutative multigroup;
  \item $(R,\cdot,1)$ is a commutative monoid;
  \item $a.0=0$ for all $a\in R$;
  \item If $c\in a+b$, then $c.d\in a.d+b.d$. Or equivalently, $(a+b).d\subseteq a.d+b.d$.
 \end{enumerate}

Note that if $a \in R$, then $0 = 0.a \in (1+ (-1)).a \subseteq 1.a + (-1).a$, thus $(-1). a = -a$.
 
 $R$ is said to be an hyperring if for $a,b,c \in R$, $a(b+c) = ab + ac$. 
 
 A multiring (respectively, a hyperring) $R$ is said to be a multidomain (hyperdomain) if it has no zero divisors. A multiring $R$ will be a 
multifield if every non-zero element of $R$ has multiplicative inverse; note that hyperfields and multifields coincide. We will use "hyperfield" since this is the prevailing terminology.
\end{definition}

 We define recursively for $n\ge 2$:
\begin{align*}
 a_1+...+a_n:=\bigcup_{d\in a_2+...+a_n}a_1+d.
\end{align*}

In particular, for a multiring $A$, with $a_1,...,a_n\in A$ and $\sigma\in S_n$, we have
 $$a_1+a_2+...+a_n=a_{\sigma(1)}+a_{\sigma(2)}+...+a_{\sigma(n)}.$$
 
 \begin{definition}[Example 2.6 in \cite{marshall2006real}]\label{defn:strangeloc}
 Fix a multiring $A$ and a multiplicative subset $S$ of $A$ such that $1\in S$. Define an equivalence relation $\sim$ 
on $A$ by $a\sim b$ if, and only if, $as=bt$ for some $s,t\in S$. Denote by $\overline a$ the equivalence class of 
$a$ and set $A/_mS=\{\overline a:a\in A\}$. Then, we define in agreement with Marshall's notation, $\overline a+\overline 
b=\{\overline c:cv\in as+bt,\,\mbox{for some }s,t,v\in S\}$, $-\overline a=\overline{-a}$, and 
$\overline{a}\overline{b}=\overline{ab}$.
 \end{definition}
 
Then  $A/_mS$ is a multiring. Moreover, if $A$ is a hyperring, the same holds for $A/_mS$. The canonical projection $\pi:A\rightarrow A/_mS$ is 
a morphism.
 \begin{proposition}[2.19 in \cite{ribeiro2016functorial}]
  Let $A,B$ be a multiring and $S\subseteq A$ a multiplicative subset of $A$. Then for every morphism 
$f:A\rightarrow B$ such that $f[S]=\{1\}$, there exist a unique morphism $\tilde f:A/_mS\rightarrow B$ such 
that the following diagram commute:
$$\xymatrix{A\ar[r]^{\pi}\ar[dr]_{f} & A/_mS\ar[d]^{!\tilde f} \\ & B}$$
where $\pi:A\rightarrow A/_mS$ is the canonical projection $\pi(a)=\overline a$.
 \end{proposition}

\begin{definition}[4.1 and 4.2 of \cite{marshall2006real}]\label{defn:mfrealreduced}
 A hyperfield $F$ is said to be \textbf{real reduced} if $a^3=a$ for all $a\in F$ and $a\in1+1$ imply $a=1$.
\end{definition}

It is proved in \cite{ribeiro2016functorial} that the category of real reduced hyperfields is equivalent to the category of reduced special groups.

\begin{definition}[7.5 and 7.6 of \cite{marshall2006real}]\label{defn:mrrealreduced}
A multiring $A$ is \textbf{real reduced} if it is semi real ($-1\notin\sum A^2$), $1\ne0$, and the following properties holds for all $a,b,c,d\in A$:
\begin{enumerate}[i -]
 \item $a^3=a$;
 \item if $c\in a+ab^2$ then $c=a$ (in other words, $a+ab^2=\{a\}$);
 \item if $c\in a^2+b^2$ and $d\in a^2+b^2$ then $c=d$ (and from (iii), we conclude that this unique element $c\in a^2+b^2$ is a square). 
\end{enumerate}
\end{definition}

\subsection{Real semigroups}

\begin{definition}[Ternary Semigroup, Definition 1.1 of \cite{dickmann2004real}]\label{defn:tsemig}
 A \textbf{ternary semigroup} (abbreviated TS) is a structure $(S,\cdot,1,0,-1)$ with individual 
constants $1,0,-1$ and a binary operation ``$\cdot$'' such that:
\begin{description}
 \item [TS1] $(S,\cdot,1)$ is a commutative semigroup with unity;
 \item [TS2] $x^3=x$ for all $x\in S$;
 \item [TS3] $-1\ne1$ and $(-1)(-1)=1$;
 \item [TS4] $x\cdot0=0$ for all $x\in S$;
 \item [TS5] For all $x\in S$, $x=-1\cdot x\Rightarrow x=0$.
\end{description}

We shall write $-x$ for $(-1)\cdot x$. The semigroup verifying conditions [TS1] and [TS2] (no extra 
constants) will be called \textbf{3-semigroups}. We denote $\mbox{Id}(S)=\{x\in S:x^2=x\}=S^2$ and $S^*=\{x\in S:x^2=1\}$.
\end{definition}

% Exemplos
\begin{example}[1.2(a) of \cite{dickmann2004real}]
  The three-element structure $\bm 3=\{1,0,-1\}$ has an obvious ternary semigroup structure.
\end{example}

Here, we will enrich the language $\{\cdot,1,0,-1\}$ with a ternary relation $D$.  We shall write $a\in D(b,c)$ instead of $D(a,b,c)$. We 
also set:
$$a\in D^t(b,c)\mbox{ if, and only if, } a\in D(b,c),\mbox{ and }-b\in D(-a,c),\mbox{ and } -c\in D(b,-a).$$
The relations $D$ and $D^t$ are called \textbf{representation} and \textbf{transversal representation} 
respectively.

\begin{definition}[Real Semigroup, 2.1 of \cite{dickmann2004real}]\label{defn:realsemigroup}
 A \textbf{real semigroup} is a ternary semigroup together with a ternary relation $D$ 
satisfying:
\begin{description}
 \item [RS0] $c\in D(a,b)$ if, and only if, $c\in D(b,a)$.
 \item [RS1] $a\in D(a,b)$.
 \item [RS2] If $a\in D(b,c)$, then $ad\in D(bd,cd)$.
 \item [RS3 (Strong Associativity)] If $a\in D^t(b,c)$ and $c\in D^t(d,e)$, then there exists $x\in 
D^t(b,d)$ such that $a\in D^t(x,e)$.
 \item [RS4] If $e\in D(c^2a,d^2b)$, then $e\in D(a,b)$.
 \item [RS5] If $ad=bd$, $ae=be$, and $c\in D(d,e)$, then $ac=bc$.
 \item [RS6] If $c\in D(a,b)$, then $c\in D^t(c^2a,c^2b)$.
 \item [RS7 (Reduction)] If $D^t(a,-b)\cap D^t(b,-a)\ne\emptyset$, then $a=b$.
 \item [RS8] If $a\in D(b,c)$, then $a^2\in D(b^2,c^2)$.
\end{description}
\end{definition}

The theory of real semigroups can be alternatively axiomatized by the transversal relation $D^t$. In this case, we define
$$c\in D(a,b)\mbox{ if, and only if, } c\in D^t(c^2a,c^2b).$$

\begin{example}[Real Semigroups and Rings, 2.2 of \cite{dickmann2004real}]\label{rsrings}
Let $A$ be a semi-real ring ($-1\notin\sum A^2$). For $a\in A$, define $\overline a:\mbox{Sper}(A)\rightarrow\bm3$ by the following rule
$$\overline a(\alpha)=\begin{cases}
                          1\mbox{ if }a\in\alpha\setminus(-\alpha) \\
                          0\mbox{ if }a\in\alpha\cap-\alpha \\
                          -1\mbox{ if }a\in(-\alpha)\cap\alpha.
                         \end{cases}.$$
                         
Let $G_A$ be the set $G_A:=\{\overline 
a:\mbox{Sper}(A)\rightarrow\bm3:a\in A\}$. Then $G_A$ with the operation induced by product in $A$ is a ternary semigroup. More generally, given a (proper) preorder $T$ 
of a ring $A$ one can relativize the definition above to $T$, by considering functions $\overline 
a$ defined on $\mbox{Sper}(A,T)=\{\alpha\in\mbox{Sper}(A):\alpha\supseteq T\}$, instead of 
$\mbox{Sper}(A)$. The corresponding ternary semigroup will be denoted $G_{A,T}$.

Now, we will equip the ternary semigroup with the representation and transversal representation 
relations given for $a,b,c\in A$, by:
\begin{align*}\overline c\in D_A(\overline a,\overline b)& \mbox{ if, and only if, for all }\alpha\in\mbox{Sper}(A)\mbox{ holds } \overline c(\alpha)=0,\mbox{ or }\overline a(\alpha)\overline  c(\alpha)=1,\mbox{ or }\overline b(\alpha)\overline c(\alpha)=1. \\
\overline c\in D^t_A(\overline a,\overline b) &\mbox{ if, and only if, for all }\alpha\in\mbox{Sper}(A) \mbox{ holds }\\
&(\overline c(\alpha)=0\mbox{ and }\overline 
a(\alpha)=\overline{-b}(\alpha)),\mbox{ or }\overline a(\alpha)\overline c(\alpha)=1,\mbox{ or }
 \overline b(\alpha)\overline c(\alpha)=1.
\end{align*}
We have that $G_A$ is a real semigroup. A similar definition with 
$\mbox{Sper}(A)$ replaced by $\mbox{Sper}(A,T)$ ($T$ a proper preordering of $A$) also endows the 
ternary semigroup $G_{A,T}$ with a structure of real semigroup.
\end{example}

\begin{theorem}[Separation Theorem, 4.4 of \cite{dickmann2004real}]\label{teo:septeo}
 Let $G$ be a real semigroup, and $a,b,c\in G$ and $X_G=Hom(G,\bm3)$. Then:
 \begin{enumerate}[i -]
  \item $a\in D_G(b,c)$ if, and only if, for all $h\in X_G$, $h(a)\in D_{\bm3}(h(b),h(c))$.
  \item $a\in D^t_G(b,c)$ if, and only if, for all $h\in X_G$, $h(a)\in D^t_{\bm3}(h(b),h(c))$.
  \item If $a\ne b$, there exists $h\in X_G$ such that $h(a)\ne h(b)$.
 \end{enumerate}
\end{theorem}

\begin{theorem}[4.14 and 4.17 of \cite{ribeiro2016functorial}]\label{teo:RS+RRM}
$ $
\begin{enumerate}[a -]
 \item Let $(G,\cdot,1,0,-1,D)$ be a real semigroup and define $+:G\times 
G\rightarrow\mathcal{P}(G)\setminus\{\emptyset\}$, $ a + b = D^t(a,b) $ and $-:G\rightarrow 
G$ by $-(g)=-1\cdot g$. Then $(G,+,\cdot,-,0,1)$ is a real reduced multiring.
 \item Let $A$ be a real reduced multiring. Then $(A,\cdot,1,0,-1,D)$ is a real semigroup, where $d\in D(a,b)$ if, and only if, $d\in  d^2a+d^2b$.
\end{enumerate}
\end{theorem}

\begin{definition}
    A real semigroup $ R $ is said to be \emph{Boolean} if $ X_R $ is a Boolean topological space. We also say that $ R $ is \emph{von Neumann} if $ R $ is Boolean and its lattice of idempotents is a Boolean algebra, i.e., for each $ x \in \mbox{Id}(R) $ there is a unique $ y \in \mbox{Id}(R) $ such that $ 1 \in D_R^t(x, y) $ and $ xy = 0 $.
\end{definition}

\begin{theorem}[Theorem 3.6 in \cite{miraglia2021boolean}]\label{rsbool-carac}
Let $ R $ be a real semigroup. The following are equivalent:

\begin{enumerate}[i -]
\item For all $ a \in R $, there exists $ u \in R^\times $ such that $ a = u a^2 $ and $ u \in D^t(-1, a) $.
\item For all $ a \in \mbox{Id}(R) $, there exists $ u \in R^\times $ such that $ a = u a $ and $ u \in D^t(-1, a) $.
\item $ R $ is a Boolean real semigroup.
\item $ R^\times $, with the induced representation from $ R $, is a reduced special group and the induced morphism $ X_R \to X_{R^\times} $ is a homeomorphism.
\end{enumerate}
\end{theorem}

\section{Generalized Marshall Quotients for real-valued continuous functions}\label{cont-rs}

Our goal is to provide explicitly calculated new classes of examples of real semigroups provenient from real-valued continuous functions. The main inspiration is Example 5.15 of \cite{roberto2021quadratic}, where is proved there that $A/_mq(A)$ is a real reduced multiring. Combining this result with Theorem \ref{rsbool-carac} we get a stronger result in Theorem \ref{str-01}, proving that the invertible elements in $A/_mq(A)$ constitute a reduced special group/real reduced hyperfield. This first example with the new ones obtained in Sections \ref{cont-rs} and \ref{der-rs} are, to the best of our knowledge, the first explicitly examples of boolean real semigroups calculated from continuous and smooth real-valued functions. 

Throughout this entire Section $X$ denotes a $T_6$ space and $A:=\mathcal C(X,\mathbb R)$ denotes the ring of real-valued continuous functions from $X$ to $\mathbb R$.

\begin{theorem}\label{str-01}
    Let $f \in \mbox{zd}(A)$. There exists a non-zero continuous extension $u \in \mbox{nzd}(A)$ of $f$ such that, $[f] = [u][f^2]$ and $[u] \in [-1] + [f]$ in $A/_m q(A)$. Moreover, $A/_m q(A)$ is a Boolean real semigroup and $(A/_m q(A))^\times\cup\{[0]\}$ is a real reduced hyperfield (reduced special group).
\end{theorem}
\begin{proof}
    Since $X$ is a $T_6$ spaces, the closed set $F = X \setminus \mbox{int}(Z(f))$ is the zero-set of some non-zero divisor continuous function $g \in A$ (Proposition \ref{t6-carac}). Let $u := f - g$ (the non-negative continuous extension of $f$). Since $g = 0$ outside $\mbox{int}(Z(f))$, we have $u|_{X \setminus Z(f)} = f|_{X \setminus Z(f)}$. As shown previously, $\mbox{int}(Z(u)) = \emptyset$, thus $u \in \mbox{nzd}(A)$, which implies $u^2 \in q(A)$. We only need to prove that $[f] = [u][f^2]$ and $[u] \in [-1] + [f]$.
    \begin{itemize}
        \item \textbf{$[f] = [u][f^2]$:} We evaluate the functions $f u^2$ and $u f^2$ globally on $X$:
    \begin{itemize}
        \item For $x \in X \setminus Z(f)$, since $u(x) = f(x)$, we have $f(x)u^2(x) = f^3(x)$ and $u(x)f^2(x) = f^3(x)$.
        \item For $x \in Z(f)$, since $f(x) = 0$, both $f(x)u^2(x)$ and $u(x)f^2(x)$ are identically $0$.
    \end{itemize}
    Thus, $fu^2 = u f^2$ in $A$, which means exactly that $f \sim_m u f^2$, providing $[f] = [u f^2] = [u][f^2]$.

        \item \textbf{$[u] \in [-1] + [f]$:} By the definition of the hyperoperation in the Marshall quotient, we must find $v, t, s \in q(A)$ such that $u v = (-1) s + f t$, which is equivalent to $s = f t - u v$.
    
    Let $\rho := u^2 + 1$. We have that $\rho$ is a positive function such that $\rho(x) - u(x) = u(x)^2 - u(x) + 1 > 0$ and $\rho(x) + u(x) = u(x)^2 + u(x) + 1 > 0$ for all $x\in X$. Define $v := (\rho - u)^2$ and $t := (\rho + u)^2$. Since $v$ and $t$ are positive functions, their zero-sets are empty, meaning $v, t \in \mbox{nzd}(A)$. Since they are squares, $v, t \in q(A)$. Now, let $s:=ft-uv$. Note that:
    \begin{align*}
        s &= f t - u v= f(\rho+u)^2 - u(\rho-u)^2.
    \end{align*}
    Substituting $f = u + g$ we obtain
    \begin{align*}
        s &= (u+g)(\rho+u)^2 - u(\rho-u)^2 \\
          &= u[(\rho+u)^2 - (\rho-u)^2] + g(\rho+u)^2 \\
          &= u[4\rho u] + g(\rho+u)^2 \\
          &= 4\rho u^2 + g(\rho+u)^2.
    \end{align*}
    Since $\rho > 0$, $u^2 \ge 0$, $g \ge 0$, and $(\rho+u)^2 \ge 0$, it follows that $s(x) \ge 0$ for all $x \in X$. In the ring $\mathcal{C}(X, \mathbb{R})$, every non-negative function is a square (the square root is continuous and we have $s(x)=(\sqrt{s(x)})^2$), hence $s \in A^2$. 
    
    To ensure $s \in q(A)$, we must check its zero-set. Since $s$ is a sum of two non-negative terms, $s(x) = 0$ implies both $4\rho(x)u^2(x) = 0$ and $g(x)(\rho(x)+u(x))^2 = 0$. Since $\rho(x)$ and $(\rho(x)+u(x))^2$ are strictly positive, this requires $u(x) = 0$ and $g(x) = 0$, which consequently means $f(x) = u(x) + g(x) = 0$. Thus, $Z(s) = Z(u) \cap Z(f)$. Since $\mbox{int}(Z(u)) = \emptyset$, we conclude that $\mbox{int}(Z(s)) = \emptyset$. 
    
    Therefore, $s \in \mbox{nzd}(A) \cap A^2 = q(A)$. With $v, t, s \in q(A)$ fulfilling the relation $u v = -s + f t$, we conclude that $[u] \in [-1] + [f]$ in the multiring $A/_m q(A)$, completing the proof that $A/_m q(A)$ is a Boolean real semigroup.
    \end{itemize}
\end{proof}

In the sequel, we discuss how to obtain a \L{}ojasiewicz inequality type for $A/_m q(A)$. In classical real analytic geometry, the \L{}ojasiewicz inequality provides an upper bound for the distance of a point to the nearest zero of a given real analytic function. For a more detailed discussion concerning origins and applications of this inequality, the reader is kindly addressed to consult \cite{colding2014lojasiewicz}, or  \cite{haraux2012some}.

M. Marshall proved that the homological consequences of this inequality---namely, the ability to continuously interpolate functions that have compatible behaviors on their zero-sets---can be captured without any metric data. More precisely, we have the following results:

\begin{theorem}[Hörmander-\ Lojasiewicz inequality, Theorem 6.8.1 in \cite{marshall1996spaces}]\label{HL-01}
Let $(X,G)$ be an abstract real spectra and let $Y\subseteq X$ be a closed subset. Suppose that $ f, g \in G $ with $ g = 0 $ on $ Y \cap Z(f) $. Then there exists $ f_1 \in D^t \langle f, g \rangle $ such that $ f_1 = f $ on $ Y $.
\end{theorem}

\begin{theorem}[Corollary 6.8.2 of \cite{marshall1996spaces}]\label{HL-02}
Let $(X,G)$ be an abstract real spectra and let $Y\subseteq X$ be a closed subset. Suppose that $f\in G$ with $f \geq 0$ on $Y \cap Z(\mathfrak a)\cap U(S^2)$ for some subsets $\mathfrak a, S\subseteq G$. Then there exists $f_1 \in G$ with $f_1 \geq 0$ on $Y$ and $f_1 = f$ on $Z(\mathfrak a) \cap U(S^2)$.
\end{theorem}

Through the categorical duality between Abstract Real Spectra (ARS) and Real Reduced Multirings ($\mathcal{MR}_{red}$), we can translate Marshall's abstract \L{}ojasiewicz Theorems into our setting. In our dictionary, the space of signs corresponds to the real spectrum $\mbox{Sper}(A)$, and the transversal representation set $D^t(f,g)$ corresponds to the multivalued sum $f+g$ in the multiring $A$.

\begin{theorem}[Adapted from \ref{HL-01}]\label{HL-03}
Let $A$ be a real reduced multiring and $Y \subseteq \mbox{Sper}(A)$ be a closed set in the spectral topology. Let $f, g \in A$. If for every ordering $\sigma \in Y$ such that $\sigma(f) = 0$ we also have $\sigma(g) = 0$, then there exists an element $f_1 \in f + g$ such that $\sigma(f_1) = \sigma(f)$ for all $\sigma \in Y$.
\end{theorem}

To formulate Theorem \ref{HL-02} in terms of multirings, we introduce the hyperalgebraic analogues of zero-sets and non-vanishing sets in the real spectrum.

\begin{definition}
Let $A$ be a real reduced multiring and consider subsets $\mathfrak{a}, S \subseteq A$. We define the \textbf{spectral zero-set} of $\mathfrak{a}$ and \textbf{the strict positivity set} of $S^2$ by:
\begin{align*}
    Z(\mathfrak{a}) &= \{\sigma \in \mbox{Sper}(A) \mid \sigma(x) = 0 \mbox{ for all } x \in \mathfrak{a}\}, \\
    U(S^2) &= \{\sigma \in \mbox{Sper}(A) \mid \sigma(s) \neq 0 \mbox{ for all } s \in S\}.
\end{align*}
\end{definition}

\begin{theorem}[Adapted from \ref{HL-02}]\label{HL-04}
Let $A$ be a real reduced multiring, $Y \subseteq \mbox{Sper}(A)$ be a closed subset, and $\mathfrak{a}, S \subseteq A$. If $f \in A$ satisfies $\sigma(f) \ge 0$ for all $\sigma \in Y \cap Z(\mathfrak{a}) \cap U(S^2)$, then there exists an element $f_1 \in A$ such that $\sigma(f_1) \ge 0$ for all $\sigma \in Y$, and $\sigma(f_1) = \sigma(f)$ for all $\sigma \in Z(\mathfrak{a}) \cap U(S^2)$.
\end{theorem}

 As an application, we discuss what the combination of these results implies for the standard topology of the real line. Therefore, we will analyze these \L{}ojasiewicz-type inequalities into the concrete ring $A = \mathcal{C}(\mathbb{R}, \mathbb{R})$. Every point $x \in \mathbb{R}$ naturally induces an ordering $\sigma_x \in \mbox{Sper}(A)$ via the sign evaluation $\sigma_x(f) = \mbox{sgn}(f(x))$. Under this dictionary, we have the following imediate consequence of Theorem \ref{HL-03}:

\begin{theorem}[\L{}ojasiewicz Interpolation on $\mathbb{R}$]\label{HL-05}
Let $F \subseteq \mathbb{R}$ be a closed set, $A:=\mathcal{C}(\mathbb{R}, \mathbb{R})$, and let $f, g \in A$. Suppose that every root of $f$ inside $F$ is also a root of $g$ (i.e., $Z(f) \cap F \subseteq Z(g) \cap F$). Then, $[f] + [g]$ (in $A/_m q(A)$) contains an element $[f_1]$ such that $\mbox{sgn}(f_1(x)) = \mbox{sgn}(f(x))$ for all $x \in F$. 
\end{theorem}

In other words, the function $f_1$ obtained in the previous Theorem perfectly preserves the sign profile of $f$ over the entire set $F$, justifying that the multivalued addition in the quotient algebra acts as an exact sign-preserving interpolator, controlling behavior of $g$ without requiring analytical bounds on derivatives.

Also, we have the following imediate consequence of Theorem \ref{HL-04}:

\begin{theorem}[Continuous Positive Extension over Boolean Domains in $\mathbb{R}$]\label{HL-06}
Let $F \subseteq \mathbb{R}$ be a closed subset. Let $Z \subseteq \mathbb{R}$ be another closed subset and $U \subseteq \mathbb{R}$ be an open subset. Suppose $f \in \mathcal{C}(\mathbb{R}, \mathbb{R})$ is non-negative on $F \cap Z \cap U$. Then, there exists a global continuous function $f_1 \in \mathcal{C}(\mathbb{R}, \mathbb{R})$ such that $f_1(x) \ge 0$ for all $x \in F$, and $\mbox{sgn}(f_1(x)) = \mbox{sgn}(f(x))$ for all $x \in Z \cap U$.
\end{theorem}

 This Theorem \ref{HL-06}, which is the translation of Theorem \ref{HL-04}, is an even more striking application which solves an extension problem over Boolean combinations of sets --- regions where the classical Tietze Extension Theorem fails.

In order to generalize these results, we will deal with generalized Marshall quotients. Indeed, these quotients could be defined for more general multiplicative subsets but for practical reasons we will only consider quotients over $q(A)$.

\subsection{The Compact Case}

Given $f,g\in A$, define a relation $\sim_{cm}$ on $A\times A$, the \textbf{compact Marshall relation}, by the rule
$$f\sim_{cm}g\mbox{ if, and only if, there exist a compact subset }K\subseteq X\mbox{ and }r,s\in q(A)\mbox{ such that }(fr)|_{X\setminus K}=(gs)|_{X\setminus K}.$$

\begin{proposition}\label{equiv-cm}
    The relation $\sim_{cm}$ is an equivalence relation.
\end{proposition}
\begin{proof}
    Let $f \in A$. We can choose any compact subset $K\subseteq X$, and $r = s = 1$. Since $1 = 1^2$ and $1 \in \mbox{nzd}(A)$, we have $1 \in q(A)$. Then, $(f \cdot 1)|_{X \setminus K} = (f \cdot 1)|_{X \setminus K}$, providing $f \sim_{cm} f$.

    Now, suppose $f \sim_{cm} g$. Then there exist a compact subset $K \subseteq X$ and $r, s \in q(A)$ such that $(fr)|_{X \setminus K} = (gs)|_{X \setminus K}$. This equation can be immediately rewritten as $(gs)|_{X \setminus K} = (fr)|_{X \setminus K}$, providing $g \sim_{cm} f$.

    Finally suppose $f \sim_{cm} g$ and $g \sim_{cm} h$. There exist compact subsets $K_1, K_2 \subseteq X$ and elements $r_1, s_1, r_2, s_2 \in q(A)$ such that
    $$(fr_1)|_{X \setminus K_1} = (gs_1)|_{X \setminus K_1} \quad \mbox{and} \quad (gr_2)|_{X \setminus K_2} = (hs_2)|_{X \setminus K_2}.$$
    Let $K = K_1 \cup K_2$. Since the finite union of compact sets is compact, $K$ is a compact subset of $X$. On the intersection $(X \setminus K_1) \cap (X \setminus K_2) = X \setminus K$, both identities hold simultaneously. Multiplying the first equation by $r_2$ and the second by $s_1$, we get:
    $$(fr_1r_2)|_{X \setminus K} = (gs_1r_2)|_{X \setminus K} \quad \mbox{and} \quad (gr_2s_1)|_{X \setminus K} = (hs_2s_1)|_{X \setminus K},$$
    which implies:
    $$(f(r_1r_2))|_{X \setminus K} = (h(s_1s_2))|_{X \setminus K}.$$
    We define $r_3 := r_1r_2$ and $s_3 := s_1s_2$. Because $q(A) = A^2 \cap \mbox{nzd}(A)$ is closed under multiplication, we have $r_3, s_3 \in q(A)$. Therefore, with the compact set $K$ and elements $r_3, s_3 \in q(A)$, we conclude that $f \sim_{cm} h$. Moreover $\sim_{cm}$ is an equivalence relation.
\end{proof}

Denote the set of equivalence classes of $A$ under the relation $\sim_{cm}$ by $A/_{cm}q(A)$. Elements in $A/_{cm}q(A)$ will be denoted by $[f]\in A/_{cm}q(A)$, for $f\in A$. In symbols:
$$A/_{cm}q(A):=A/\sim_{cm}=\{[a]:a\in A\}.$$
We call $A/_{cm}q(A)$ by \textbf{the compact Marshall quotient} of $A$ by $q(A)$. Our aim is to prove that $A/_{cm}q(A)$ has a structure of Boolean real semigroup. We will do it in the following sequence of results.

\begin{proposition}\label{cong-cm}
    If $f,f',g,g'\in A$ are such that $f\sim_{cm}f'$ and $g\sim_{cm}g'$, then:
    \begin{enumerate}[a -]
        \item $fg \sim_{cm} f'g'$.
        \item $\{[h]: hr \sim_{cm} fs+gt \mbox{ for some } r,s,t\in q(A)\} = \{[h]: hr' \sim_{cm} f's'+g't' \mbox{ for some } r',s',t'\in q(A)\}$.
    \end{enumerate}
\end{proposition}

\begin{proof}
$ $
\begin{enumerate}[a -]
    \item By definition, there exist compact sets $K_1, K_2 \subseteq X$ and elements $r_1, s_1, r_2, s_2 \in q(A)$ such that
    $$(fr_1)|_{X \setminus K_1} = (f's_1)|_{X \setminus K_1} \quad \mbox{and} \quad (gr_2)|_{X \setminus K_2} = (g's_2)|_{X \setminus K_2}.$$
    Let $K = K_1 \cup K_2$, which is a compact subset of $X$. On the intersection $X \setminus K$, both identities hold simultaneously. Multiplying them pointwise yields:
    $$(fr_1)(gr_2)|_{X \setminus K} = (f's_1)(g's_2)|_{X \setminus K}.$$
    Rearranging the terms, we get: 
    $$(fg)(r_1r_2)|_{X \setminus K} = (f'g')(s_1s_2)|_{X \setminus K}.$$ 
    Since $q(A) = A^2 \cap \mbox{nzd}(A)$ is closed under multiplication, $r_1r_2, s_1s_2 \in q(A)$. Therefore, with the compact set $K$ and these new weights, we conclude that $fg \sim_{cm} f'g'$.
    
    \item By symmetry, it suffices to prove the inclusion $\subseteq$. Let $[h]$ be an element of the left-hand set. Then, there exist $r,s,t \in q(A)$ such that $hr \sim_{cm} fs + gt$. 
    
    By opening the definition of $\sim_{cm}$, there exist a compact set $K_h \subseteq X$ and elements $\alpha, \beta \in q(A)$ such that:
    $$(hr\alpha)|_{X \setminus K_h} = ((fs+gt)\beta)|_{X \setminus K_h} = (fs\beta + gt\beta)|_{X \setminus K_h}.$$
    
    Since $f \sim_{cm} f'$ and $g \sim_{cm} g'$, there exist compact sets $K_f, K_g \subseteq X$ and elements $a, a', b, b' \in q(A)$ such that:
    $$(fa)|_{X \setminus K_f} = (f'a')|_{X \setminus K_f} \quad \mbox{and} \quad (gb)|_{X \setminus K_g} = (g'b')|_{X \setminus K_g}.$$
    
    We multiply the first equation by $(ab) \in q(A)$:
    $$(hr\alpha \cdot ab)|_{X \setminus K_h} = (fs\beta \cdot ab + gt\beta \cdot ab)|_{X \setminus K_h}.$$
    We can rewrite the right side to explicitly expose the terms $(fa)$ and $(gb)$:
    $$(h \cdot (r\alpha ab))|_{X \setminus K_h} = ((fa) \cdot (s\beta b) + (gb) \cdot (t\beta a))|_{X \setminus K_h}.$$
    
    Let $K_{new} = K_h \cup K_f \cup K_g$, which is compact. On $X \setminus K_{new}$, all three local identities hold simultaneously. We can substitute $(fa)$ with $(f'a')$ and $(gb)$ with $(g'b')$:
    $$(h \cdot (r\alpha ab))|_{X \setminus K_{new}} = ((f'a') \cdot (s\beta b) + (g'b') \cdot (t\beta a))|_{X \setminus K_{new}}.$$
    Rearranging the terms, we obtain:
    $$(h \cdot (r\alpha ab))|_{X \setminus K_{new}} = (f' \cdot (a's\beta b) + g' \cdot (b't\beta a))|_{X \setminus K_{new}}.$$
    
    Now, we define $r' := r\alpha ab$, $s' := a's\beta b$, and $t' := b't\beta a$. Since $q(A)$ is closed under multiplication, $r', s', t' \in q(A)$. The equation above states that $h r'$ and $f's' + g't'$ are strictly equal outside the compact set $K_{new}$, which implies that $hr' \sim_{cm} f's' + g't'$. This establishes that $[h]$ belongs to the right-hand set, concluding the proof.
\end{enumerate}
\end{proof}

By Proposition \ref{cong-cm}, the rules below are well defined for $f,g\in A$:
\begin{align*}
    [f] \cdot [g] &:= [fg] \\
    [f] + [g] &:= \{[h] : hr \sim_{cm} fs + gt \mbox{ for some } r,s,t \in q(A)\}.
\end{align*}

\begin{theorem}\label{mr-cm}
    The structure $(A/_{cm}q(A),+,\cdot,[0],[1])$ is a multiring.
\end{theorem}
\begin{proof}
    By Proposition \ref{cong-cm}, the operation $[f] \cdot [g] = [fg]$ is well-defined. Since $A$ is a commutative ring with unity $1$ and zero $0$, $(A/_{cm}q(A), \cdot, [1])$ is a commutative monoid, and $[f] \cdot [0] = [0]$ for all $[f] \in A/_{cm}q(A)$. The commutativity of the hyperoperation $+$ is also inherited from the commutativity of $A$ (since $fs+gt = gt+fs$). It remains to verify the multigroup axioms and the distributivity law.

    \begin{itemize}
        \item \textbf{Zero element:} We must show that $[f] + [0] = \{[f]\}$. Let $[h] \in [f] + [0]$. By definition, there exist $r, s\in q(A)$, a compact $K \subseteq X$, and $\alpha, \beta \in q(A)$, such that $(hr\alpha)|_{X \setminus K} = (fs\beta)|_{X \setminus K}$. This can be rewritten as $h(r\alpha) \sim_{cm} f(s\beta)$. Since $q(A)$ is closed under multiplication, $r\alpha \in q(A)$ and $s\beta \in q(A)$, which by definition implies $[h] = [f]$. 
    Conversely, choosing $1, 1, 1 \in q(A)$, we have $f \cdot 1 \sim_{cm} f \cdot 1 + 0 \cdot 1$, hence $[f] \in [f] + [0]$.

        \item \textbf{Reversibility:} We must show that $[h] \in [f] + [g]$ if, and only if, $ [f] \in [h] + [-g]$. Suppose $[h] \in [f] + [g]$. There exist $r, s, t \in q(A)$ such that $hr \sim_{cm} fs + gt$. We can rearrange the terms in the ring $A$ to obtain:
    $$fs \sim_{cm} hr - gt = hr + (-g)t.$$
    Defining $r' := s$, $s' := r$, and $t' := t$ (all of which belong to $q(A)$), the relation becomes $f r' \sim_{cm} h s' + (-g) t'$. This is exactly the defining condition for $[f] \in [h] + [-g]$.

        \item \textbf{Associativity:} We must show that $([f] + [g]) + [k] = [f] + ([g] + [k])$, which is equivalent to prove the inclusion $\subseteq$. Let $[h] \in ([f] + [g]) + [k]$. By definition, there exists $[u] \in [f] + [g]$ with $[h] \in [u] + [k]$, which provide compact subsets $K_1, K_2 \subseteq X$ and $r_1, s_1, t_1, r_2, s_2, t_2 \in q(A)$ such that:
    \begin{align}
        (u r_1)|_{X \setminus K_1} &= (f s_1 + g t_1)|_{X \setminus K_1}, \\
        (h r_2)|_{X \setminus K_2} &= (u s_2 + k t_2)|_{X \setminus K_2}.
    \end{align}
    Let $K = K_1 \cup K_2$, which is a compact subset of $X$. We multiply equation (2) by $r_1 \in q(A)$ and evaluate it on $X \setminus K$, where both local identities hold simultaneously:
    $$(h r_2 r_1)|_{X \setminus K} = (u r_1 s_2 + k t_2 r_1)|_{X \setminus K}.$$
    Substituting the term $(u r_1)$ using equation (1), we obtain:
    \begin{align*}
        (h r_1 r_2)|_{X \setminus K} &= ((f s_1 + g t_1)s_2 + k t_2 r_1)|_{X \setminus K}= (f s_1 s_2 + g t_1 s_2 + k t_2 r_1)|_{X \setminus K}.
    \end{align*}
    Let $v := g(t_1 s_2) + k(t_2 r_1)\in A$. Because $q(A)$ is closed under multiplication, $t_1 s_2 \in q(A)$ and $t_2 r_1 \in q(A)$. The identity $v \cdot 1 = g(t_1 s_2) + k(t_2 r_1)$ implies that $v \sim_{cm} g(t_1 s_2) + k(t_2 r_1)$. This means that $[v] \in [g] + [k]$.
    
    Substituting the newly defined function $v$ back into our evaluated expression for $h$, we get:
    $$(h (r_1 r_2))|_{X \setminus K} = (f (s_1 s_2) + v \cdot 1)|_{X \setminus K}.$$
    Define $R := r_1 r_2$, $S := s_1 s_2$, and $T := 1$. Of course, $R, S, T \in q(A)$. The equation becomes $(h R)|_{X \setminus K} = (f S + v T)|_{X \setminus K}$, which by definition states that $[h] \in [f] + [v]$. 
    
    Since $[v] \in [g] + [k]$, we conclude that $[h] \in [f] + ([g] + [k])$, completing the proof of associativity.

        \item \textbf{Distributivity:} Let $[f],[g],[h],[k]\in A/_{cm}q(A)$ with $[h] \in [f] + [g]$. We need to show $[hk] \in [fk] + [gk]$.
    Since $[h] \in [f] + [g]$, there exist $r, s, t \in q(A)$ such that $hr \sim_{cm} fs + gt$.
    Multiplying both sides of the equivalence by $k$ we get that $(hr)k \sim_{cm} (fs + gt)k$ implies $(hk)r \sim_{cm} (fk)s + (gk)t$. This new relation directly states that $[hk] \in [fk] + [gk]$.
    \end{itemize}

    Therefore, $(A/_{cm}q(A), +, \cdot, [0], [1])$ satisfies all the axioms of Definition \ref{defn:multiring}, making it a multiring.
\end{proof}

\begin{theorem}\label{rrm-cm}
    If $X$ is not compact, the multiring $(A/_{cm}q(A),+,\cdot,[0],[1])$ is a real reduced multiring.
\end{theorem}
\begin{proof}
    We must verify the axioms of Definition \ref{defn:mrrealreduced}.
    \begin{itemize}
        \item \textbf{$[1] \ne [0]$ and Semi-reality:} Suppose $[1] = [0]$. Then there exist a compact set $K \subseteq X$ and $r, s \in q(A)$ such that $(1 \cdot r)|_{X \setminus K} = (0 \cdot s)|_{X \setminus K}$, which means $r(x) = 0$ for all $x \in X \setminus K$. Since $X$ is a $T_6$ space and not compact, $X \setminus K$ is a non-empty open set. Thus, $\mbox{int}(Z(r)) \supseteq X \setminus K \ne \emptyset$, contradicting $r \in \mbox{nzd}(A)$. Therefore, $[1] \ne [0]$.
        
        Now, suppose $[-1] \in [f_1]^2 + \dots + [f_n]^2$. There exist a compact $K \subseteq X$ and $r,s_1,...,s_n \in q(A)$ such that $(-r)|_{X \setminus K} = (f_1^2s_1 + \dots + f_n^2s_n)|_{X \setminus K}$. Since $r,s_1,...,s_n \in q(A)$, they are squares of non-zero-divisors, making them strictly positive on a dense set. The left side is $\le 0$ and the right side is $\ge 0$, which forces both sides to be identically $0$ on $X \setminus K$. Thus $r(x) = 0$ on $X \setminus K$, leading to the exact same contradiction regarding the non-empty interior of $Z(r)$. Hence, $A/_{cm}q(A)$ is semi-real.

        \item  \textbf{$[f]^3 = [f]$:} By Proposition \ref{nz-extension}, for any $f \in A$ there exists a non-zero divisor continuous extension $u$ such that $f u^2 = u f^2 = f^3$ holds. This implies $f \cdot u^2 = f^3 \cdot 1$. Since $1 \in q(A)$ and $u \in \mbox{nzd}(A)$ implies $u^2 \in q(A)$, we have $f \sim_{cm} f^3$. Thus, $[f] = [f]^3$.

        \item \textbf{$[h] \in [f] + [f][g]^2$ implies $[h] = [f]$:} Suppose $[h] \in [f] + [fg^2]$. By definition, there exist a compact $K \subseteq X$ and $r,s,t \in q(A)$ such that 
    $$(hr)|_{X \setminus K} = (fs + fg^2t)|_{X \setminus K} = (f(s+g^2t))|_{X \setminus K}.$$
    Let $v = s + g^2t$. Since $s, t \in q(A)$, we have $v \ge 0$. Furthermore, $Z(v) = Z(s) \cap Z(g^2t) \subseteq Z(s)$. Since $s \in \mbox{nzd}(A)$, $\mbox{int}(Z(s)) = \emptyset$, which forces $\mbox{int}(Z(v)) = \emptyset$. Since $v\ge0$, we conclude $v \in q(A)$ (again, the square-root is a continuous operation and $v(x)=(\sqrt{v(x)})^2$ for all $x\in X$). The equation becomes $(hr)|_{X \setminus K} = (fv)|_{X \setminus K}$, which is the exact definition of $h \sim_{cm} f$. Thus, $[h] = [f]$.

    \item \textbf{If $[h] \in [f]^2 + [g]^2$ then $[h] = [f^2 + g^2]$:}  Suppose $[h] \in [f^2] + [g^2]$. There exist a compact $K \subseteq X$ and $r,s,t \in q(A)$ such that $(hr)|_{X \setminus K} = (f^2s + g^2t)|_{X \setminus K}$. Let $w := f^2 + g^2$. We want to show that $h \sim_{cm} w$.
    
    Since $r, s, t \in q(A)$, they are non-negative and their zero-sets have empty interiors. Let $U$ be the open set $(X \setminus K) \setminus (Z(r) \cup Z(s) \cup Z(t))$. Because $X$ is a $T_6$ space, the union of closed nowhere-dense sets is nowhere-dense, making $U$ an open dense subset of $X \setminus K$. 
    For $x \in U$, we have $r(x)>0$, $s(x)>0$, and $t(x)>0$. The local relation gives $h(x)r(x) = f(x)^2s(x) + g(x)^2t(x) \ge 0$. Since $r(x) > 0$, it follows that $h(x) \ge 0$ for all $x \in U$. Let $H := |h|$. Because $h$ is non-negative on $U$, we have the equality:
    $$h(x)w(x) = H(x)w(x) \quad \mbox{for all } x \in U.$$
    Since $U$ is dense in $X \setminus K$ and the functions are continuous, this equality extends to the whole complement, giving $(hw)|_{X \setminus K} = (Hw)|_{X \setminus K}$.
    
     By Propositions \ref{nz-extension} and \ref{t6-carac}, there exist non-negative continuous functions $\gamma_w, \gamma_H \in A$ such that $Z(\gamma_w) = X \setminus \mbox{int}(Z(w))$ and $Z(\gamma_H) = X \setminus \mbox{int}(Z(H))$. 
    Define $p_w := w + \gamma_w$ and $p_H := H + \gamma_H$. Both functions are globally non-negative and have empty interior zero-sets. In the ring $A = \mathcal{C}(X, \mathbb{R})$, every non-negative function is a square, hence $p_w, p_H \in A^2 \cap \mbox{nzd}(A) = q(A)$. We evaluate the product $h \cdot p_w$ on $X \setminus K$:
    $$(h p_w)|_{X \setminus K} = (hw + h\gamma_w)|_{X \setminus K}.$$
    For any $x \in \mbox{int}(Z(w))$, $w=0$ implies $f=g=0$. Over the open set $\mbox{int}(Z(w)) \cap (X \setminus K)$, the local relation gives $hr = 0$. Since $\mbox{int}(Z(r)) = \emptyset$, this forces $h = 0$ on that open set. Thus, the product $h\gamma_w$ is globally zero on $X \setminus K$, meaning $(h p_w)|_{X \setminus K} = (hw)|_{X \setminus K}$. Similarly, evaluating $w \cdot p_H$ on $X \setminus K$ we get:
    $$(w p_H)|_{X \setminus K} = (wH + w\gamma_H)|_{X \setminus K}.$$
    
    For any $x \in \mbox{int}(Z(H))$, $H$ is identically $0$ on an open neighborhood $V$. Over $V \cap (X \setminus K)$, $h = 0$, so the local relation yields $f^2s + g^2t = 0$. Since $s, t \ge 0$, this implies $f^2s = 0$ and $g^2t = 0$. Because $\mbox{int}(Z(s)) = \mbox{int}(Z(t)) = \emptyset$, it forces $f = g = 0$, meaning $w = 0$ on $V \cap (X \setminus K)$. Thus, $w\gamma_H$ is globally zero on $X \setminus K$, giving $(w p_H)|_{X \setminus K} = (wH)|_{X \setminus K}$.
    
    Since we already proved that $(hw)|_{X \setminus K} = (Hw)|_{X \setminus K}$, we directly obtain:
    $$(h p_w)|_{X \setminus K} = (w p_H)|_{X \setminus K}.$$
    With $p_w, p_H \in q(A)$ and the compact set $K$, this is exactly the definition of $h \sim_{cm} w$, proving that $[h] = [w] = [f^2 + g^2]$.
    \end{itemize}

    Having verified all the axioms, we conclude that $(A/_{cm}q(A),+,\cdot,[0],[1])$ is a real reduced multiring.
\end{proof}

Finally, we arrive at our main result in this Subsection.

\begin{theorem}\label{boolean-cm}
    If $X$ is not compact, the real reduced multiring $(A/_{cm}q(A),+,\cdot,[0],[1])$ is (when considered as a real semigroup) a Boolean real semigroup. Moreover $(A/_{cm}q(A))^\times\cup\{[0]\}$ is a real reduced hyperfield (reduced special group).
\end{theorem}
\begin{proof}
    We will use the characterization provided by Theorem \ref{rsbool-carac}. It suffices to show that for every class $[f] \in R$, there exists an invertible element $[u] \in R^\times$ such that $[f] = [u][f]^2$ and $[u] \in [-1] + [f]$.

    Let $f \in A$. We analyze two cases:
    
    \textbf{Case 1:} $f \in \mbox{nzd}(A)$. 
    Since $f$ is a non-zero-divisor, we have $f^2 \in q(A)$ and $[f^2] = [1]$ in $R$. Choose $u = f$. Since $[u^2] = [1]$, $[u]$ is its own inverse, meaning $[u] \in R^\times$ and $[u][f]^2 = [f][1] = [f]$. 
    
    For the second condition, we must check that $[f] \in [-1] + [f]$. Let $\rho := f^2 + 1$, $r := (\rho - f)^2$, and $t := (\rho + f)^2$. Since $\rho \pm f = f^2 \pm f + 1 > 0$ are strictly positive, $r$ and $t$ have empty zero-sets, meaning $r, t \in \mbox{nzd}(A) \cap A^2 = q(A)$. Now, let $s\in A$ be defined by the following:
    $$s := f t - f r = f[(\rho+f)^2 - (\rho-f)^2] = f[4\rho f] = 4\rho f^2.$$
    Since $\rho \ge 1$ and $f^2 \ge 0$, $s$ is a globally non-negative function (hence a square in $A$). Furthermore, $Z(s) = Z(f)$. Because $f \in \mbox{nzd}(A)$, we have $\mbox{int}(Z(s)) = \emptyset$, meaning $s \in q(A)$. 
    The global identity $f \cdot r = (-1) \cdot s + f \cdot t$ provides $[f] \in [-1] + [f]$.

    \textbf{Case 2:} $f \in \mbox{zd}(A)$. 
    By Proposition \ref{nz-extension}, there exists a continuous non-zero divisor extension $u \in \mbox{nzd}(A)$ of $f$. Since $u \in \mbox{nzd}(A)$, $u^2 \in q(A)$, meaning $[u]$ is invertible in $R$ (with inverse $[u]$). In Theorem \ref{str-01}, we established the existence of global identities in the ring $A$ for this exact extension $u$. Specifically, we proved that $fu^2 = uf^2$ (globally on $X$). Furthermore, we constructed explicitly $v, s, t \in q(A)$ such that $uv = -s + ft$ (also globally on $X$). 
    The first global identity yields $f u^2 \sim_{cm} u f^2$, which implies $[f][u^2] = [u][f^2]$. Since $[u^2] = [1]$, we get $[f] = [u][f]^2$.
    The second global identity $uv \sim_{cm} (-1)s + ft$ provides that $[u] \in [-1] + [f]$.

    Since in both cases the required unit $[u]$ exists, condition (i) of Theorem \ref{rsbool-carac} is satisfied, proving that $A/_{cm}q(A)$ is a Boolean real semigroup. Moreover, by condition (iv) of Theorem \ref{rsbool-carac}, the set of units $(A/_{cm}q(A))^\times$ equipped with the induced transversal representation forms a reduced special group or equivalently, a real reduced hyperfield.
\end{proof}

As a direct consequence of the fact that $A/_{cm}q(A)$ is a real reduced multiring, the abstract \L{}ojasiewicz-type inequalities apply to this quotient. When we restrict our attention to the concrete case $X = \mathbb{R}$, we obtain the following result (which is also an immediate consequence of Theorem \ref{HL-03}).

\begin{theorem}[Asymptotic \L{}ojasiewicz Interpolation on $\mathbb{R}$]
Let $A = \mathcal{C}(\mathbb{R}, \mathbb{R})$ and consider the compact Marshall quotient $A/_{cm}q(A)$. Suppose $F \subseteq \mathbb{R}$ is a closed subset, and $f, g \in A$ are functions such that there exists $M > 0$ such that for all $x \in F$ with $|x| > M$, $f(x) = 0$ implies $g(x) = 0$ (in other words, their asymptotic zero-loci on $F$ are compatible). Then, the hyperoperation $[f] + [g]$ in $A/_{cm}q(A)$ contains an element $[f_1]$ that perfectly preserves the asymptotic sign profile of $f$. That is, there exists $M' > 0$ such that $\mbox{sgn}(f_1(x)) = \mbox{sgn}(f(x))$ for all $x \in F$ with $|x| > M'$.
\end{theorem}

As an immediate consequence of Theorem \ref{HL-04} we have the following result.

\begin{theorem}[Asymptotic Positive Extension]
Let $Z, F \subseteq \mathbb{R}$ be closed subsets and $U \subseteq \mathbb{R}$ be an open subset. The intersection $W = F \cap Z \cap U$ represents a Boolean combination of sets (e.g., a disjoint union of half-open intervals extending to infinity). If a function $f \in \mathcal{C}(\mathbb{R}, \mathbb{R})$ is such that $f(x) \ge 0$ for $x \in W \setminus [-M, M]$ (in other words, $f$ is asymptotically non-negative on $W$), then there exists a global continuous function $f_1 \in \mathcal{C}(\mathbb{R}, \mathbb{R})$ such that, for sufficiently large $|x|$, $f_1(x) \ge 0$ for all $x \in F$, and $\mbox{sgn}(f_1(x)) = \mbox{sgn}(f(x))$ for all $x \in W$.
\end{theorem}

\begin{remark}
Classical topological extension theorems (like Tietze) struggle with non-closed domains such as $W$. By shifting the problem to the hyperalgebraic structure of $A/_{cm}q(A)$, the axioms of real reduced multirings automatically guarantee the existence of a global function that extends this scattered asymptotic positivity, overriding the local topological constraints of the real line.
\end{remark}

\subsection{The Closed Case}

Now, we want to define the generalized Marshall quotient considering closed subsets (instead of compact ones). In this sense, we need some adaptation in order to obtain non trivial quotients.

Let $\mathcal F$ be a family of closed subsets such that $\mathcal F$ is closed under finite unions and $X\notin\mathcal F$. Given $f,g\in A$, define a relation $\sim_{clm}$ on $A\times A$, the \textbf{closed Marshall relation}, by the rule
$$f\sim_{clm}g\mbox{ if, and only if, there exist a closed subset }F\in\mathcal F\mbox{ and }r,s\in q(A)\mbox{ such that }(fr)|_{X\setminus F}=(gs)|_{X\setminus F}.$$

\begin{proposition}\label{equiv-clm}
    The relation $\sim_{clm}$ is an equivalence relation.
\end{proposition}
\begin{proof}
    Similar to Proposition \ref{equiv-cm}.
\end{proof}

Denote the set of equivalence classes of $A$ under the relation $\sim_{clm}$ by $A/_{clm}q(A)$. Elements in $A/_{clm}q(A)$ will be denoted by $[f]\in A/_{clm}q(A)$, for $f\in A$. In symbols:
$$A/_{clm}q(A):=A/\sim_{clm}=\{[a]:a\in A\}.$$
We call $A/_{clm}q(A)$ by \textbf{the closed Marshall quotient} of $A$ by $q(A)$. 

Note that, if $X$ is not compact, $f\sim_{cm}g$ if, and only if, $f\sim_{clm}g$ with $\mathcal F=\{F\subseteq X:F\mbox{ is compact}\}$.

Our aim is to prove that $A/_{clm}q(A)$ has a structure of Boolean real semigroup. We will do it in the following sequence of results.

\begin{proposition}\label{cong-clm}
    If $f,f',g,g'\in A$ are such that $f\sim_{clm}f'$ and $g\sim_{clm}g'$, then:
    \begin{enumerate}[a -]
        \item $fg \sim_{clm} f'g'$.
        \item $\{[h]: hr \sim_{clm} fs+gt \mbox{ for some } r,s,t\in q(A)\} = \{[h]: hr' \sim_{clm} f's'+g't' \mbox{ for some } r',s',t'\in q(A)\}$.
    \end{enumerate}
\end{proposition}
\begin{proof}
    Similar to Proposition \ref{cong-cm}.
\end{proof}

By Proposition \ref{cong-clm}, the rules below are well defined for $f,g\in A$:
\begin{align*}
    [f] \cdot [g] &:= [fg] \\
    [f] + [g] &:= \{[h] : hr \sim_{clm} fs + gt \mbox{ for some } r,s,t \in q(A)\}.
\end{align*}

\begin{theorem}\label{mr-clm}
    The structure $(A/_{clm}q(A),+,\cdot,[0],[1])$ is a multiring.
\end{theorem}
\begin{proof}
    Similar to Theorem \ref{mr-cm}.
\end{proof}

\begin{theorem}\label{rrm-clm}
    The multiring $(A/_{clm}q(A),+,\cdot,[0],[1])$ is a real reduced multiring.
\end{theorem}
\begin{proof}
    The verification of the axioms follows the proof for $A/_{cm}q(A)$ (Theorem \ref{rrm-cm}). The structural requirement that $X \notin \mathcal{F}$ plays the exact same role as non-compactness did previously: it ensures that for any closed set $F \in \mathcal{F}$, the complement $X \setminus F$ is non-empty. 
    
    Since $X$ is a $T_6$ space and $F$ is closed, $X \setminus F$ is a non-empty open set. If $[1] = [0]$, there would exist $F \in \mathcal{F}$ and $r, s \in q(A)$ such that $r(x) = 0$ for all $x \in X \setminus F$. This implies $\mbox{int}(Z(r)) \supseteq X \setminus F \ne \emptyset$, which contradicts $r \in \mbox{nzd}(A)$. Thus, $[1] \ne [0]$ and the quotient does not collapse. 
    
    The remaining axioms ($[f]^3=[f]$ and the reduction of sums of squares) follow verbatim by restricting the global abstract Marshall quotient identities to $X \setminus F$.
\end{proof}

\begin{theorem}\label{boolean-clm}
    The real reduced multiring $(A/_{clm}q(A),+,\cdot,[0],[1])$ is (when considered as a real semigroup) a Boolean real semigroup. Moreover $(A/_{clm}q(A))^\times\cup\{[0]\}$ is a real reduced hyperfield (reduced special group).
\end{theorem}
\begin{proof}
    As shown in the compact (Theorem \ref{boolean-cm}), the continuous non-zero extension of any function $f \in A$ (Proposition \ref{nz-extension}) yields an element $u \in \mbox{nzd}(A)$ such that the algebraic identities $f u^2 = u f^2$ and $uv = -s + ft$ hold globally on $X$ for specific $v, s, t \in q(A)$. Then $[u]$ is invertible, $[f] = [u][f]^2$, and $[u] \in [-1] + [f]$. By Theorem \ref{rsbool-carac}, $(A/_{clm}q(A), +, \cdot, [0], [1])$ is a Boolean real semigroup and its units form a real reduced hyperfield.
\end{proof}

To illustrate the analytical power of the closed Marshall quotient, consider the space $X = \mathbb{R}$ and define $\mathcal{F}$ as the family of all \textbf{nowhere dense closed sets} (i.e., closed sets with empty interiors). Since the finite union of nowhere dense closed sets remains nowhere dense, and $\mathbb{R}$ itself clearly has a non-empty interior ($\mathbb{R} \notin \mathcal{F}$), this family constitutes a proper ideal of subsets. 

By the Baire Category Theorem, the complement $\mathbb{R} \setminus F$ of any set $F \in \mathcal{F}$ is a dense open set. Consequently, equivalence in $A/_{clm}q(A)$ captures the \emph{generic behavior} of continuous functions, idealing out "meager" topological noise. Applying the abstract hyperalgebraic machinery to this specific quotient yields Baire-generic versions of \L{}ojasiewicz inequalities.

The following Theorem is an immediate consequence of Theorem \ref{HL-03}.

\begin{theorem}[Generic \L{}ojasiewicz Interpolation]
Let $A = \mathcal{C}(\mathbb{R}, \mathbb{R})$ and consider $A/_{clm}q(A)$ over the family of nowhere dense closed sets. Let $Y \subseteq \mathbb{R}$ be a closed set, and $f, g \in A$. Suppose that $Z(f) \cap Y \subseteq Z(g) \cap Y$. Then, the hyperoperation $[f] + [g]$ contains an element $[f_1]$ that acts as a generic sign-preserving interpolator: there exists a dense open set $\Omega \subseteq \mathbb{R}$ such that $\mbox{sgn}(f_1(x)) = \mbox{sgn}(f(x))$ for all $x \in Y \cap \Omega$.
\end{theorem}

Similarly, the following Theorem is an immediate consequence of Theorem \ref{HL-04}.

\begin{theorem}[Generic Positive Extension over Boolean Domains]
Let $Z, Y \subseteq \mathbb{R}$ be closed sets and $U \subseteq \mathbb{R}$ be an open set. Let $W = Y \cap Z \cap U$ be a Boolean combination of sets. If a continuous function $f \in \mathcal{C}(\mathbb{R}, \mathbb{R})$ is non-negative on $W$, then there exists a global continuous function $f_1 \in \mathcal{C}(\mathbb{R}, \mathbb{R})$ and a dense open set $\Omega \subseteq \mathbb{R}$ such that $f_1(x) \ge 0$ for all $x \in Y \cap \Omega$, and $\mbox{sgn}(f_1(x)) = \mbox{sgn}(f(x))$ for all $x \in W \cap \Omega$.
\end{theorem}

\subsection{The Bounded Case}

Finally, we define the generalized Marshall quotient considering only bounded subsets. For this, in this section we assume that $X$ is a metric space.

Given $f,g\in A$, define a relation $\sim_{bm}$ on $A\times A$, the \textbf{bounded Marshall relation}, by the rule
$$f\sim_{bm}g\mbox{ if, and only if, there exist a bounded subset }M\subseteq X\mbox{ and }r,s\in q(A)\mbox{ such that }(fr)|_{X\setminus M}=(gs)|_{X\setminus M}.$$

\begin{proposition}\label{equiv-bm}
    The relation $\sim_{bm}$ is an equivalence relation.
\end{proposition}
\begin{proof}
    Similar to Proposition \ref{equiv-cm}.
\end{proof}

Denote the set of equivalence classes of $A$ under the relation $\sim_{bm}$ by $A/_{bm}q(A)$. Elements in $A/_{bm}q(A)$ will be denoted by $[f]\in A/_{bm}q(A)$, for $f\in A$. In symbols:
$$A/_{bm}q(A):=A/\sim_{bm}=\{[a]:a\in A\}.$$
We call $A/_{bm}q(A)$ by \textbf{the bounded Marshall quotient} of $A$ by $q(A)$. Our aim is to prove that $A/_{bm}q(A)$ has a structure of Boolean real semigroup. We will do it in the following sequence of results.

\begin{proposition}\label{cong-bm}
    If $f,f',g,g'\in A$ are such that $f\sim_{bm}f'$ and $g\sim_{bm}g'$, then:
    \begin{enumerate}[a -]
        \item $fg \sim_{bm} f'g'$.
        \item $\{[h]: hr \sim_{bm} fs+gt \mbox{ for some } r,s,t\in q(A)\} = \{[h]: hr' \sim_{bm} f's'+g't' \mbox{ for some } r',s',t'\in q(A)\}$.
    \end{enumerate}
\end{proposition}
\begin{proof}
    Similar to Proposition \ref{cong-cm}.
\end{proof}

By Proposition \ref{cong-bm}, the rules below are well defined for $f,g\in A$:
\begin{align*}
    [f] \cdot [g] &:= [fg] \\
    [f] + [g] &:= \{[h] : hr \sim_{bm} fs + gt \mbox{ for some } r,s,t \in q(A)\}.
\end{align*}

\begin{theorem}\label{mr-bm}
    The structure $(A/_{bm}q(A),+,\cdot,[0],[1])$ is a multiring.
\end{theorem}
\begin{proof}
    Similar to Theorem \ref{mr-cm}.
\end{proof}

\begin{theorem}\label{rrm-bm}
    If $X$ is an unbounded metric space, then the multiring $(A/_{bm}q(A),+,\cdot,[0],[1])$ is a real reduced multiring.
\end{theorem}
\begin{proof}
    The structural requirement that $X$ is unbounded plays a critical role. It ensures that for any bounded set $M\subseteq X$, the complement $X \setminus M$ is non-empty. 
    
    If $[1] = [0]$, there would exist a bounded set $M$ and $r, s \in q(A)$ such that $r(x) = 0$ for all $x \in X \setminus M$. Since $M$ is bounded, replacing $M$ by its closure if necessary, we obtain $\mbox{int}(Z(r)) \ne \emptyset$, contradicting $r \in \mbox{nzd}(A)$. Thus, $[1] \ne [0]$. The remaining axioms ($[f]^3=[f]$ and squares reduction) are similar to Theorem \ref{rrm-cm}.
\end{proof}

\begin{theorem}\label{boolean-bm}
    The real reduced multiring $(A/_{bm}q(A),+,\cdot,[0],[1])$ is (when considered as a real semigroup) a Boolean real semigroup. Moreover $(A/_{bm}q(A))^\times\cup\{[0]\}$ is a real reduced hyperfield (reduced special group).
\end{theorem}
\begin{proof}
    As established previously, the continuous non-zero divisor extension $u \in \mbox{nzd}(A)$ of any $f \in A$ yields the global identities $f u^2 = u f^2$ and $uv = -s + ft$ for suitable $v, s, t \in q(A)$. This proves that $[f] = [u][f]^2$ and $[u] \in [-1] + [f]$ in $A/_{bm}q(A)$. By Theorem \ref{rsbool-carac}, the quotient is a Boolean real semigroup and its units form a real reduced hyperfield.
\end{proof}

\begin{remark}
We might observe a topological redundancy when restricting the space to Euclidean domains $X = \mathbb{R}^n$. By the Heine-Borel Theorem, every closed and bounded set in $\mathbb{R}^n$ is compact. Since any bounded set is contained within a compact closure, the families of compact sets and bounded sets generate the exact same topological ideal at infinity. Consequently, for $A = \mathcal{C}(\mathbb{R}^n, \mathbb{R})$, the compact Marshall quotient and the bounded Marshall quotient are canonically isomorphic:
$$A/_{cm}q(A) \cong A/_{bm}q(A).$$

However, this isomorphism breaks down in infinite-dimensional metric spaces, such as general Banach spaces. In infinite dimensions, the closed unit ball is bounded but not compact. Therefore, "escaping all compact sets" (the compact quotient) and "escaping to metric infinity" (the bounded quotient) measure two entirely distinct asymptotic phenomena.
\end{remark}

In the case where $X$ is an unbounded metric space, the bounded Marshall quotient $A/_{bm}q(A)$ ideals out all behavior on bounded sets, capturing purely the \emph{metric asymptotic} properties of continuous functions as the distance from an origin tends to infinity. We now apply the abstract \L{}ojasiewicz-type inequalities to this quotient.

The following Theorem is an immediate consequence of Theorem \ref{HL-03}.

\begin{theorem}[Metric-Asymptotic \L{}ojasiewicz Interpolation]
Let $X$ be an unbounded metric space, $A = \mathcal{C}(X,\mathbb{R})$, and consider the bounded Marshall quotient $A/_{bm}q(A)$.  
Let $F \subseteq X$ be a closed set, and let $f, g \in A$. Suppose that there exists a bounded set $M_0 \subset X$ such that for all $x \in F \setminus M_0$, $f(x)=0$ implies $g(x)=0$. Then the hyperoperation $[f] + [g]$ in $A/_{bm}q(A)$ contains an element $[f_1]$ such that there exists a bounded set $M \subset X$ with
\[\mbox{sgn}(f_1(x)) = \mbox{sgn}(f(x)) \qquad \mbox{for all } x \in F \setminus M .
\]
\end{theorem}

\begin{remark}
This theorem states that two continuous functions can be algebraically combined via the multivalued sum in $A/_{bm}q(A)$ such that the resulting function $f_1$ inherits the exact sign pattern of $f$ on the ``ends'' of the space $X$, completely disregarding any oscillatory or irregular behavior that $g$ might exhibit in bounded regions.  
In the finite‑dimensional case $X = \mathbb{R}^n$, the bounded quotient reduces to the compact Marshall quotient (by Heine–Borel), and the statement becomes identical to the asymptotic interpolation already described in Section~3.1. The true novelty appears in infinite‑dimensional Banach spaces, where \emph{bounded} and \emph{compact} are distinct; here the bounded quotient isolates the behavior as $\|x\|\to\infty$ in the norm sense, a genuinely new hyperalgebraic tool.
\end{remark}

Similarly, the following Theorem is an immediate consequence of Theorem \ref{HL-04}.

\begin{theorem}[Metric-Asymptotic Positive Extension]
Let $X$ be an unbounded metric space and $A = \mathcal{C}(X,\mathbb{R})$.  
Let $F,Z \subseteq X$ be closed subsets and $U \subseteq X$ be an open subset. Consider the Boolean combination $W = F \cap Z \cap U$.   If a function $f \in A$ is such that there exists a bounded set $M_0 \subset X$ with $f(x) \ge 0$ for all $x \in W \setminus M_0$, then there exists a global continuous function $f_1 \in A$ and a bounded set $M \subset X$ such that $f_1(x) \ge 0$ for all $x \in F \setminus M$, and $\mbox{sgn}(f_1(x)) = \mbox{sgn}(f(x))$ for all $x \in W \setminus M$.
\end{theorem}

\section{Generalized Marshall Quotients for real-valued differentiable functions}\label{der-rs}

Now, we discuss the generalized Marshall quotient for the ring $\mathcal C^k(\mathbb R^n,\mathbb R)$ of $k$-differentiable real-valued functions from $\mathbb R^n$ to $\mathbb R$ ($n\in\mathbb N$ with $n\ge1$). We could replace $\mathbb R^n$ for a finite-dimension Banach space but for practical reasons, we stick to $X=\mathbb R^n$.

Even though we could perform the Marshall quotient $D/_mq(D)$,
the resulting multiring is not real reduced in general.
The obstruction stems from the fact that in rings of differentiable functions,
a non‑negative function need not be a square of a $C^k$ function.
While for continuous functions every $h \ge 0$ admits a continuous square root,
the analogous statement for $C^k$ spaces is false:
already for $k=1$ there exist non‑negative $C^1$ functions that are not squares of $C^1$ functions.
For instance, define $f:\mathbb R\to\mathbb R$ by
\[
f(x)=
\begin{cases}
x^{4}\bigl(2+\sin\frac{1}{x}\bigr), & x\neq 0,\\[4pt]
0, & x=0.
\end{cases}
\]
One checks that $f$ is $C^1$ on $\mathbb R$ and $f(x)\ge 0$ for all $x$.
Its pointwise square root is $\sqrt{f(x)}=x^{2}\sqrt{2+\sin\frac{1}{x}}$ (for $x\neq0$) and $\sqrt{f(0)}=0$.
A short computation shows that the derivative of $\sqrt{f}$ is unbounded near $0$ due to the oscillatory term
$\cos\frac{1}{x}$, so $\sqrt{f}\notin C^1(\mathbb R)$.
Since any $C^1$ function whose square equals $f$ must, up to sign, coincide with $\sqrt{f}$,
$f$ cannot be a square in $C^1(\mathbb R)$. Consequently,
$q(D)=D^{2}\cap\mbox{nzd}(D)$ is strictly smaller than
the set of all non‑negative non‑zero‑divisors.

The proofs that $A/_mq(A)$ is a real reduced multiring for $A=\mathcal C(X,\mathbb R)$
heavily exploit the fact that every non‑negative continuous function belongs to $A^2$;
in particular, the weight $s$ constructed in the verification of the axiom $[f]^3=[f]$
is shown to be a square precisely because it is non‑negative.
In the $C^k$ setting, a non‑negative $C^k$ function may fail to have a $C^k$ square root,
so the required algebraic certificates cannot be manufactured inside $q(D)$,
and $D/_mq(D)$ will typically violate $a^3=a$ or the reduction of sums of squares
(both of which are necessary for real reducedness).

This is the very reason why the generalized Marshall quotients
(modulo compact, closed or bounded supports) are introduced:
they relax the definition of the hyperoperations such that the equalities
only need to hold outside a suitable ``exceptional'' set,
while the witnesses are still taken from $q(D)$.
The structure of $T_6$ spaces (or $\mathbb R^n$) then guarantees that outside such a set
functions can be assumed strictly positive and therefore do admit smooth roots,
allowing the construction of the missing certificates and restoring the real reduced axioms.

\subsection{The Compact Case}

To deal with this analytical obstruction, we must restrict our domain to a proper subring $D$ where the generalized Marshall relation can "absorb" the local irregularities into the equivalence relation itself. 

\begin{definition}
    We define the subring $D \subseteq \mathcal{C}^k(\mathbb{R}^n, \mathbb{R})$ of \textbf{eventually constant functions} as:
    $$D := \{f \in \mathcal{C}^k(\mathbb{R}^n, \mathbb{R}):\mbox{there exists a compact } K \subset \mathbb{R}^n, c \in \mathbb{R} \mbox{ such that } f(x) = c \mbox{ for all } x \in \mathbb{R}^n \setminus K\}.$$
\end{definition}

It is immediate that $D$ is a commutative ring with unity, since the sum and product of two eventually constant functions remain eventually constant outside the union of their respective compact sets.

Given $f,g\in D$, define a relation $\sim_{cm}$ on $D\times D$, the \textbf{compact Marshall relation}, by the rule
$$f\sim_{cm}g\mbox{ if, and only if, there exist a compact }K\subseteq \mathbb R^n\mbox{ and }r,s\in q(D)\mbox{ such that }(fr)|_{\mathbb R^n\setminus K}=(gs)|_{\mathbb R^n\setminus K}.$$

\begin{proposition}\label{equiv-cm-der}
    The relation $\sim_{cm}$ is an equivalence relation.
\end{proposition}
\begin{proof}
    Similar to Proposition \ref{equiv-cm}.
\end{proof}

Denote the set of equivalence classes of $D$ under the relation $\sim_{cm}$ by $D/_{cm}q(D)$. Elements in $D/_{cm}q(D)$ will be denoted by $[f]\in D/_{cm}q(D)$, for $f\in D$. In symbols:
$$D/_{cm}q(D):=D/\sim_{cm}=\{[a]:a\in D\}.$$
We call $D/_{cm}q(D)$ by \textbf{the compact Marshall quotient} of $D$ by $Q(D)$. Our aim is to prove that $D/_{cm}q(D)$ has a structure of Boolean real semigroup. We will do it in the following sequence of results.

\begin{proposition}\label{cong-cm-der}
    If $f,f',g,g'\in D$ are such that $f\sim_{cm}f'$ and $g\sim_{cm}g'$, then:
    \begin{enumerate}[a -]
        \item $fg \sim_{cm} f'g'$.
        \item $\{[h]: hr \sim_{cm} fs+gt \mbox{ for some } r,s,t\in q(D)\} = \{[h]: hr' \sim_{cm} f's'+g't' \mbox{ for some } r',s',t'\in q(D)\}$.
    \end{enumerate}
\end{proposition}
\begin{proof}
    Similar to Proposition \ref{cong-cm}
\end{proof}

By Proposition \ref{cong-cm-der}, the rules below are well defined for $f,g\in A$:
\begin{align*}
    [f] \cdot [g] &:= [fg] \\
    [f] + [g] &:= \{[h] : hr \sim_{cm} fs + gt \mbox{ for some } r,s,t \in q(D)\}.
\end{align*}

\begin{theorem}\label{mr-cm-der}
    The structure $(D/_{cm}q(D),+,\cdot,[0],[1])$ is a multiring.
\end{theorem}
\begin{proof}
    Similar to Theorem \ref{mr-cm}.
\end{proof}

\begin{theorem}\label{rrm-cm-der}
    The multiring $(D/_{cm}q(D),+,\cdot,[0],[1])$ is a real reduced multiring.
\end{theorem}
\begin{proof}
    We must verify the axioms of Definition \ref{defn:mrrealreduced}.
    \begin{itemize}
        \item \textbf{$[1] \neq [0]$ and Semi-reality:} Suppose $[1] = [0]$. There exist a compact $K \subseteq \mathbb{R}^n$ and $r, s \in q(D)$ such that $1 \cdot r = 0 \cdot s$ on $\mathbb{R}^n \setminus K$, meaning $r(x) = 0$ outside $K$. Since $\mathbb{R}^n$ is unbounded, $\mathbb{R}^n \setminus K$ contains a non-empty open set, meaning $\mbox{int}(Z(r)) \neq \emptyset$, which contradicts $r \in \mbox{nzd}(D)$. Thus $[1] \neq [0]$. Semi-reality follows from the exact same spatial contradiction for sums of squares.

        \item \textbf{$[f]^3 = [f]$:} For any $f \in D$, $f$ is equal to a constant $c \in \mathbb{R}$ outside some compact $K_0$. 
        If $c = 0$, then $f(x) = 0$ outside $K_0$. We can choose global constant functions $r=1, s=1 \in q(D)$, yielding $f \cdot 1 = f^3 \cdot 1$ outside $K_0$.
        If $c \neq 0$, we choose the global constant functions $r = c^2$ and $s = 1$. Since $c \neq 0$, both $r$ and $s$ are strictly positive constants, meaning they are globally non-zero-divisors and perfect squares in $\mathcal{C}^k$, so $r, s \in q(D)$. Outside $K_0$, $f(x) \cdot c^2 = c \cdot c^2 = c^3 = f(x)^3 \cdot 1$. Thus, $(f r)|_{\mathbb{R}^n \setminus K_0} = (f^3 s)|_{\mathbb{R}^n \setminus K_0}$, proving $f \sim_{cm} f^3$, and $[f] = [f]^3$.

        \item \textbf{If $[h] \in [f] + [f][g]^2$ then $[h] = [f]$:} Suppose $[h] \in [f] + [fg^2]$. By definition, there exist a compact $K_1$ and $r,s,t \in q(D)$ such that $hr = f(s+g^2t)$ outside $K_1$. Since $f,g,h,r,s,t \in D$, we can enlarge $K_1$ to a compact $K_2$ such that outside $K_2$, all these functions are strictly constant: $C_f, C_g, C_h, C_r, C_s, C_t$. Because $r, s, t \in q(D)$, they cannot be zero on an open set, so their constants $C_r, C_s, C_t$ must be strictly positive. 
        The local relation becomes $C_h C_r = C_f(C_s + C_g^2 C_t)$. Let $C_v = C_s + C_g^2 C_t > 0$. We define global constant functions $R = C_r$ and $S = C_v$. Since both are strictly positive constants, $R, S \in q(D)$. Outside $K_2$, $h(x) \cdot R(x) = C_h C_r$ and $f(x) \cdot S(x) = C_f C_v$. Since $C_h C_r = C_f C_v$, the equality $h R = f S$ holds outside $K_2$, meaning $h \sim_{cm} f$, and $[h] = [f]$.

        \item \textbf{If $[h] \in [f]^2 + [g]^2$ then $[h] = [f^2 + g^2]$:} Suppose $[h] \in [f^2] + [g^2]$. There exist a compact $K_1$ and $r,s,t \in q(D)$ such that $hr = f^2s + g^2t$ outside $K_1$. Let $w = f^2 + g^2$. Outside a sufficiently large $K_2$, all functions evaluate to constants: $C_h C_r = C_f^2 C_s + C_g^2 C_t$. Again, $C_r, C_s, C_t > 0$, implying $C_h \ge 0$.
        We want to show $h \sim_{cm} w$. Outside $K_2$, $w(x)$ is the constant $C_w = C_f^2 + C_g^2 \ge 0$. 
        If $C_f = C_g = 0$, then $C_w = 0$, and $C_h C_r = 0$ which implies $C_h = 0$. Using global weights $R=1, S=1 \in q(D)$, $h R = w S = 0$ outside $K_2$.
        If $C_f^2 + C_g^2 > 0$, then $C_w > 0$. Furthermore, $C_h = (C_f^2 C_s + C_g^2 C_t) / C_r \ge 0$. We define global strictly positive constant functions $R = C_w$ and $S = C_h$ (if $C_h > 0$, else $R=1, S=1$), so $R, S \in q(D)$. Outside $K_2$, $h(x) R(x) = C_h C_w$ and $w(x) S(x) = C_w C_h$, making them equal. Thus, $[h] = [f^2 + g^2]$.
    \end{itemize}
    Having verified all axioms without extracting non-constant smooth roots, $(D/_{cm}q(D),+,\cdot,[0],[1])$ is a real reduced multiring.
\end{proof}

\begin{theorem}\label{boolean-cm-der}
    The real reduced multiring $(D/_{cm}q(D),+,\cdot,[0],[1])$ is (when considered as a real semigroup) a Boolean real semigroup. Moreover $(D/_{cm}q(D))^\times\cup\{[0]\}$ is a real reduced hyperfield.
\end{theorem}
\begin{proof}
    We use the characterization provided by Theorem~\ref{rsbool-carac}. It suffices to show that for every $[f] \in D/_{cm}q(D)$, there exists $[u] \in (D/_{cm}q(D))^\times$ such that $[f] = [u][f]^2$ and $[u] \in [-1] + [f]$.
    
    Let $f \in D$. Since $f$ is eventually constant, there exists a compact $K$ such that $f(x) = c \in \mathbb{R}$ for all $x \notin K$. We analyze two cases based on $c$.
    
    \medskip
    \noindent\textbf{Case 1:} $c \neq 0$. This implies $f \in \mbox{nzd}(D)$. 
    Choose $u = f$. Since $[f^2] = [1]$, $[u]$ is its own inverse, meaning $[u] \in (D/_{cm}q(D))^\times$ and $[u][f]^2 = [f]^3 = [f]$. 
    To show $[f] \in [-1] + [f]$, we need weights $r,s,t \in q(D)$ such that $f r = -s + f t$ outside some compact set. We define global strictly positive constant functions:
    $$r(x) := (c^2 + 1 - c)^2, \quad t(x) := (c^2 + 1 + c)^2, \quad s(x) := 4(c^2+1) c^2.$$
    Since $c \neq 0$, $r, s, t$ are strictly positive constants, meaning $r,s,t \in q(D)$. The algebraic identity $c \cdot r = -s + c \cdot t$ holds exactly for these constants, meaning $(f r)|_{\mathbb{R}^n \setminus K} = (-s + f t)|_{\mathbb{R}^n \setminus K}$. Thus, $[f] \in [-1] + [f]$.
    
    \medskip
    \noindent\textbf{Case 2:} $c = 0$. This means $f(x) = 0$ outside $K$, so $f \in \mbox{zd}(D)$. Since $f(x) = 0$ outside $K$, the interior of its zero-set, $\mbox{int}(Z(f))$, contains $\mathbb{R}^n \setminus K$. The complement $F = \mathbb{R}^n \setminus \mbox{int}(Z(f))$ is a closed subset of $K$, meaning $F$ is compact. 
    By standard smooth bump function theory, there exists a smooth non-negative function $g \ge 0$ such that $Z(g) = F$, and $g(x) = 1$ for all $x$ outside some larger compact set $K' \supseteq K$.
    We define $u \in D$ globally by $u(x) = f(x) - g(x)$. 
    By construction, 
    $$Z(u) = Z(f) \cap Z(g) = Z(f) \cap F = \partial Z(f)$$ 
    which has empty interior, so $u \in \mbox{nzd}(D)$. 
    Outside $K'$, $f(x) = 0$ and $g(x) = 1$, so $u(x) = -1$. Since $u$ is eventually $-1$, $u^2$ is eventually $1$. Using $1, 1 \in q(D)$, the identity $u^2 \cdot 1 = 1 \cdot 1$ holds outside $K'$, meaning $[u^2] = [1]$ and $[u] \in (D/_{cm}q(D))^\times$.
    To show $[f] = [u][f]^2$, note that outside $K'$, $f(x) = 0$ and $u(x) = -1$. The identity $f \cdot 1 = u f^2 \cdot 1$ ($0 = 0$) holds outside $K'$.
    To show $[u] \in [-1] + [f]$, we need $u r \sim_{cm} -s + f t$. Outside $K'$, $f = 0$ and $u = -1$. Then, choosing $r = 1, s = 1, t = 1 \in q(D)$, we get $[u] \in [-1] + [f]$.
    
    \medskip
    In all cases, the required unit $[u]$ exists, proving that $D/_{cm}q(D)$ is a Boolean real semigroup.
\end{proof}

The compact Marshall quotient $D/_{cm}q(D)$ for the ring $D$ of functions that are $C^k$ outside
some compact set captures the \emph{asymptotic $C^k$ behavior} of real‑valued functions on 
$\mathbb R^n$. Two functions become equivalent if they coincide outside a compact set up to 
multiplication by strictly positive non‑zero‑divisors that are squares in $D$. In particular, 
irregularities confined to a compact region (such as points where a function fails to be $C^k$,
or oscillatory singularities) are completely idealed out; only the smooth tail at infinity 
survives in the quotient. Because $D/_{cm}q(D)$ is a real reduced multiring, the abstract \L{}ojasiewicz-type inequalities of Section~\ref{cont-rs} apply verbatim to this quotient.

\subsection{The Closed Case}

This case proceeds essentially analogous to the compact case. 

For the closed case, we consider a family $\mathcal F$ of closed subsets such that $\mathcal F$ is closed under finite unions and $X\notin\mathcal F$. After that, we consider the subring
$D_{cl} \subseteq \mathcal{C}^k(\mathbb{R}^n, \mathbb{R})$ defined by:
    $$D_{cl} := \{f \in \mathcal{C}^k(\mathbb{R}^n, \mathbb{R}):\mbox{there exist } F\in \mathcal F, c \in \mathbb{R} \mbox{ such that } f(x) = c \mbox{ for all } x \in \mathbb{R}^n \setminus F\}.$$

Given $f,g\in D_{cl}$, define a relation $\sim_{clm}$ on $D_{cl}\times D_{cl}$, the \textbf{closed Marshall relation}, by the rule
$$f\sim_{clm}g\mbox{ if, and only if, there exist }F\in\mathcal F \mbox{ and }r,s\in q(D_{cl})\mbox{ such that }(fr)|_{\mathbb R^n\setminus F}=(gs)|_{\mathbb R^n\setminus F}.$$

\begin{proposition}\label{equiv-clm-der}
    The relation $\sim_{clm}$ is an equivalence relation.
\end{proposition}
\begin{proof}
    Similar to Proposition \ref{equiv-cm-der}
\end{proof}

Denote the set of equivalence classes of $D_{cl}$ under the relation $\sim_{clm}$ by $D_{cl}/_{clm}q(D_{cl})$. Elements in $D_{cl}/_{clm}q(D_{cl})$ will be denoted by $[f]\in D_{cl}/_{clm}q(D_{cl})$, for $f\in D$. In symbols:
$$D_{cl}/_{clm}q(D_{cl}):=D_{cl}/\sim_{clm}=\{[a]:a\in D_{cl}\}.$$
We call $D_{cl}/_{clm}q(D_{cl})$ by \textbf{the closed Marshall quotient} of $D_{cl}$ by $Q(D_{cl})$. Similarly we prove that $D_{cl}/_{clm}q(D_{cl})$ has a structure of Boolean real semigroup. We sketch it in the results below.

With a similar proof to the one given in Proposition \ref{cong-cm-der}, the equivalence relation $\sim_{clm}$ induces a well-defined hyperoperation. We safely endow $D_{cl}/_{clm}q(D_{cl})$ with the generalized Marshall quotient structure by defining:
\begin{align*}
    [f] \cdot [g] &:= [fg] \\
    [f] + [g] &:= \{[h] : hr \sim_{clm} fs + gt \mbox{ for some } r,s,t \in q(D_{cl})\}.
\end{align*}

\begin{theorem}\label{mr-clm-der}
    The structure $(D_{cl}/_{clm}q(D_{cl}),+,\cdot,[0],[1])$ is a multiring.
\end{theorem}
\begin{proof}
    Similar to Theorem \ref{mr-cm-der}
\end{proof}

\begin{theorem}\label{rrm-clm-der}
    The multiring $(D_{cl}/_{clm}q(D_{cl}),+,\cdot,[0],[1])$ is a real reduced multiring.
\end{theorem}
\begin{proof}
    Similar to Theorem \ref{rrm-cm-der}.
\end{proof}

\begin{theorem}\label{boolean-clm-der}
    The real reduced multiring $(D_{cl}/_{clm}q(D_{cl}),+,\cdot,[0],[1])$ (when considered as a real semigroup) is a Boolean real semigroup. Moreover $(D_{cl}/_{clm}q(D_{cl}))^\times\cup\{[0]\}$ is a real reduced hyperfield (reduced special group).
\end{theorem}
\begin{proof}
    Similar to Theorem \ref{boolean-cm-der}
\end{proof}

Because $D_{cl}/_{clm}q(D_{cl})$ is a real reduced multiring, the abstract \L{}ojasiewicz-type inequalities of Section~\ref{cont-rs} apply verbatim to this quotient.

\subsection{Comments on the general case}

The reader may wonder why we did not simply take $D = \mathcal C^k(\mathbb R^n,\mathbb R)$
throughout Section~\ref{der-rs}. The reason is that the resulting generalized Marshall quotients
fail to be real reduced in general, for subtle reasons that we now explain.
This discussion also illuminates why restricting the functions to be eventually constant is
essential in the compact and closed cases, and why the
``bounded'' variant collapses to the compact one in Euclidean spaces.

\begin{enumerate}
    \item \textbf{The global $C^k$ case.}
For $D = \mathcal C^k(\mathbb R^n,\mathbb R)$, the classical Marshall quotient
$D/_m q(D)$ is not real reduced.
The obstruction is that a non-negative $C^k$ function need not be a square in $D$,
so $q(D) = D^2 \cap \mbox{nzd}(D)$ is strictly smaller than the set of
non-negative non-zero-divisors.
One might hope that the generalized quotients $D/_{cm}q(D)$, $D/_{clm}q(D)$ and
$D/_{bm}q(D)$ could deal with this difficulty because their hyperoperations only
require equalities to hold outside a compact/closed/bounded set.
However, the proofs that these quotients are real reduced heavily rely on algebraic
certificates (square weights) that must exist globally in $q(D)$. For $D = \mathcal C^k(\mathbb R^n,\mathbb R)$,
this breaks down: the function $h(x) = \sin^2 x$ on $\mathbb R$ is $C^\infty$, $h \ge 0$,
$\mbox{int}(Z(h)) = \emptyset$, yet the pointwise square root $|\sin x|$
is not $C^1$. Thus the central tool for constructing the algebraic certificates collapses,
and the multirings are not real reduced.

    \item \textbf{The ring of functions that are $C^k$ outside a compact set.}
A natural relaxation is to take
$$D = \{f : \mathbb R^n \to \mathbb R \mid \mbox{there exists a compact } K 
\mbox{ such that } f|_{\mathbb R^n\setminus K} \in \mathcal C^k(\mathbb R^n\setminus K,\mathbb R)\}.$$
Here, the problem persists if the zero-set $Z(h)$ is unbounded. The unbounded part of $Z(h)$ 
can be a closed set with empty interior, and the square-root function will not be $C^k$ there. 
For instance, $h(x) = \sin^2(1/x) e^{-1/x^2}$ (extended by $0$ at $0$) is
$C^\infty$, non-negative, with empty interior zero-set, but admits no $C^k$
square root on any neighborhood of $0$ whose square is $h$.

    \item \textbf{ Enforcing eventually constant functions.}
    The solution adopted is to restrict $D$ to those $C^k$ functions that are \emph{eventually constant} outside the exceptional sets (compact, closed, or bounded). For such functions, one can evaluate them purely as real constants outside the domain of restriction. This completely bypasses the extraction of square roots of generic smooth functions. Because strictly positive real constants are trivially perfect squares globally in $\mathcal{C}^k$, the algebraic identities of Marshall's theory hold perfectly on these asymptotic weights, mollifying the analytical obstruction and unlocking the real reduced structure.

    \item \textbf{The bounded quotient.} In $\mathbb R^n$ every bounded set has compact closure, so a function that is eventually constant outside a bounded set automatically satisfies the compact case conditions. Consequently the bounded Marshall quotient coincides with the compact one. This is a Euclidean phenomenon: in infinite-dimensional Banach spaces bounded $\neq$ compact, and the two quotients would genuinely differ. We leave the exploration of that setting for future work.
\end{enumerate}

\section{Final Remarks}

In this paper, we introduced generalized Marshall quotients $A/_{cm}q(A)$, $A/_{clm}q(A)$, and $A/_{bm}q(A)$ for rings of continuous functions on suitable topological spaces. By demonstrating that they are real reduced multirings and Boolean real semigroups, we provided explicitly calculated examples linking topological function spaces to abstract real spectra.

These results provide a large new class of explicitly calculated examples of Boolean real semigroups and reduced special groups arising naturally from analysis and topology. In particular, they shed light on the open problem mentioned in the introduction: "\emph{when does the group of units of a real semigroup form a reduced special group?}" by providing new computable examples of real semigroups where the answer for this question is "yes".

The extension of the theory to differentiable functions required an analysis of the square-root obstruction, i.e., the existence of non-negative smooth functions whose square root is not differentiable. To deal with this, we restricted our domain to the ring of \emph{eventually constant functions}. By doing that, the generalized Marshall quotients evaluate the equivalence relations precisely on the regions where the functions behave as purely real constants. This algebraic strategy bypasses the extraction of complex smooth roots, recovering the full strength of the real reduced and Boolean properties while moving the irregular differential behavior into the exceptional compact or closed sets.

The applications presented in this paper illustrate how the abstract hyperalgebraic identities translate into topological and differential statements. The \L{}ojasiewicz-type inequalities show that the Boolean real semigroup structure acts as a bridge replacing derivative bounds by algebraic manipulation of sums of squares.

Several directions remain open for future investigation. First, the infinite-dimensional metric case: as highlighted by the bounded quotient, escaping compact sets and escaping bounded sets measure entirely distinct asymptotic phenomena in general Banach spaces. 

Second, the structural limits of this theory in the realm of real analytic functions ($\mathcal{C}^\omega$). The generalized asymptotic quotient topologically implodes in $\mathcal{C}^\omega$ due to the Identity Principle (an eventually constant analytic function must be globally constant). Algebraically, the classical Marshall quotient is obstructed by Hilbert's 17th Problem: non-negative analytic functions are not necessarily squares, requiring sums of squares or rational functions. Adapting the multiplicative set to capture analytic spectra remains an open problem.

Third, dropping the Boolean requirement opens the door to the study of local rings in differential geometry. If one considers $A = \mathcal{C}_0^k$ as the ring of germs of $\mathcal{C}^k$ functions at the origin, the classical Marshall quotient $A/_m q(A)$ forms a real reduced multiring. However, it is not a Boolean one: the condition $[x] = [u][x^2]$ lacks a solution in the quotient, as the required unit $u \sim 1/x$ diverges at the origin and does not belong to $\mathcal{C}_0^k$.

Finally, the functorial aspects of the generalized Marshall quotients (e.g., the behavior under pull-backs by proper maps) deserve a systematic treatment. We hope that the ideas presented here will stimulate further interaction between real algebra, hyperfields, and the geometry of function spaces.

\vspace{0.5cm}

\noindent {\bf Acknowledgments:}
The author was supported by the S\~ao Paulo Research Foundation (FAPESP, Brazil), thematic project {\em Rationality, logic and probability -- RatioLog}, grant 2020/16353-3 and by a post-doctoral grant from FAPESP, grant 2024/18577-7.

\bibliographystyle{plain}
\bibliography{one-for-all}
\end{document}